\documentclass{amsart}
\usepackage{amsfonts,amscd,amsthm,amsgen,amsmath,amssymb}
\usepackage[all]{xy}
\usepackage{epsfig,color}
\usepackage[vcentermath]{youngtab}
\newtheorem{theorem}{Theorem}[section]
\newtheorem{lemma}[theorem]{Lemma}
\newtheorem{corollary}[theorem]{Corollary}
\newtheorem{proposition}[theorem]{Proposition}

\theoremstyle{definition}
\newtheorem{definition}[theorem]{Definition}
\newtheorem{example}[theorem]{Example}

\theoremstyle{remark}
\newtheorem{remark}[theorem]{Remark}

\numberwithin{equation}{section}

\begin{document}

\setlength\parskip{0.5em plus 0.1em minus 0.2em}

\title[Constraints on embedded spheres and real projective planes]{Constraints on embedded spheres and real projective planes in 4-manifolds from Seiberg--Witten theory}
\author{David Baraglia}

\address{School of Mathematical Sciences, The University of Adelaide, Adelaide SA 5005, Australia}

\email{david.baraglia@adelaide.edu.au}


\date{\today}

\begin{abstract}

We calculate the Seiberg--Witten invariants of branched covers of prime degree, where the branch locus consists of embedded spheres. Aside from the formula itself, our calculations give rise to some new constraints on configurations of embedded spheres in $4$-manifolds. Using similar methods, we also obtain new constraints on embeddings of real projective planes and spheres with a cusp singularity. Moreover, we show that the existence of certain configurations of surfaces would give rise to $4$-manifolds of non-simple type. Our proof makes use of equivariant Seiberg--Witten invariants as well as a gluing formula for the relative Seiberg--Witten invariants of $4$-manifolds with positive scalar curvature boundary.

\end{abstract}

\maketitle




\section{Introduction}

In this paper we study Seiberg--Witten theory on cyclic branched covers of $4$-manifolds where the branch locus is assumed to be a union of spheres (in the case of double covers we also consider real projective planes). Our main objective was to obtain a formula expressing the Seiberg--Witten invariants of the branched cover in terms of Seiberg--Witten invariants of the base and we succeed in doing this. Recently the author defined the {\em equivariant Seiberg--Witten invariants} for a $4$-manifold with a finite group action \cite{bar}. Our methods actually provide a calculation of the equivariant Seiberg--Witten invariants of the branched cover with respect to the cyclic covering action. The calculation of the ordinary Seiberg--Witten invariant follows from this by applying a forgetful map. In the process of studying Seiberg--Witten theory of branched covers, we obtain some non-trivial consequences for configurations of embedded spheres or embedded real projective planes.

\subsection{Configurations of embedded surfaces}\label{sec:config}

Let us recall the adjunction inequality of Seiberg--Witten theory (for the version given here, see \cite[Corollary 1.7]{os2}). Let $X$ be a compact, oriented, smooth $4$-manifold of simple type with $b_+(X) > 1$ and let $S \subset X$ be a smooth embedded compact oriented surface of genus $g > 0$. If $\mathfrak{s}$ is a spin$^c$-structure on $X$ with $SW(X , \mathfrak{s}) \neq 0$, then
\[
| \langle c(\mathfrak{s}) , [S] \rangle | + [S]^2 \le 2g-2.
\]
The above result does {\em not} hold in the case $g=0$. Consider for example a sphere of self-intersection $-1$ in a blowup. However if $[S]$ can be represented by a sphere, then it can also be represented by a torus by adding a handle, hence the $g=1$ case can be applied, giving
\[
| \langle c(\mathfrak{s}) , [S] \rangle | + [S]^2 \le 0.
\]
In particular, $[S]^2 \le 0$. In fact, if $[S]$ is non-torsion then one has $[S]^2 < 0$ \cite[Lemma 5.1]{fs}. We are interested in the bordeline case of this inequality where $|\langle c(\mathfrak{s}) , [S] \rangle | + [S]^2 = 0$. The previously mentioned example of a sphere of self-intersection $-1$ in a blowup is an example of this. Our first main result essentially says that if $[S]$ is divisible by a prime $p$ and $SW(X , \mathfrak{s})$ is not divisible by $p$, then the borderline case does not occur:

\begin{theorem}\label{thm:1sph}
Let $p$ be a prime number. Let $X$ be a compact, oriented, smooth $4$-manifold with $H_1(X ; \mathbb{Z}_p) = 0$ and $b_+(X) > 1$. Suppose $X$ is of simple type and let $\mathfrak{s}$ be a spin$^c$-structure such that $SW(X , \mathfrak{s}) \neq 0 \; ({\rm mod} \; p)$. Let $S \subset X$ be a smoothly embedded sphere with $[S] = 0 \; ({\rm mod} \; p)$ and $[S]$ non-torsion. Then $| \langle c(\mathfrak{s}) , [S] \rangle | + [S]^2 < 0$.
\end{theorem}

The next result is an extension of this to the case of multiple disjoint embedded spheres $S_1, \dots , S_r \subset X$. Applying the adjunction inequality to each sphere individually gives $\sum_i |\langle c(\mathfrak{s}) , [S_i] \rangle | + [S_i]^2 \le 0$, but in fact a stronger result is true under a suitable divisibility condition:

\begin{theorem}\label{thm:2sph}
Let $(X,\mathfrak{s})$ be as in Theorem \ref{thm:1sph}. Suppose that $S_1, \dots , S_r \subset X$ are disjoint embedded spheres whose Poincar\'e dual classes are linearly depenent in $H^2(X ; \mathbb{Z}_p)$, but any $r-1$ of them are linearly independent in $H^2(X ; \mathbb{Z}_p)$. Then 
\[
\sum_i | \langle c(\mathfrak{s}) , [S_i] \rangle | + [S_i]^2 \le 4 - 2r.
\]
\end{theorem}

Our next result concerns embedings of real projective planes in $4$-manifolds. Let $S \subset X$ be an embedded real projective plane. Then $S$ has a mod $2$ Poincar\'e dual class $[S] \in H^2(X ; \mathbb{Z}_2)$ and an Euler number $e(S) \in \mathbb{Z}$, the twisted Euler class of the normal bundle evaluated against the fundamental class of $S$ (which is a homology class with coefficients in the orientation local system of $S$). Alternatively, $e(S)$ is the self-intersection number of $S$. Unlike the oriented case, the self-intersection number is not determined by $[S]$, although its value mod $2$ is since $e(S) = [S]^2 \; ({\rm mod} \; 2)$. It is an interesting and challenging problem to determine what pairs $([S] , e(S)) \in H^2(X ; \mathbb{Z}_2) \times \mathbb{Z}$ can be realised by a smoothly embedded real projective plane. One restriction due to Rokhlin \cite{rok} is as follows. Assume that $H_1(X ; \mathbb{Z}_2) = 0$ and $[S] = 0 \in H^2(X ; \mathbb{Z}_2)$. Then
\[
-2 - 4b_-(X) \le e(S) \le 2 + 4b_+(X).
\]
Another restriction, due to Guillou--Marin \cite{gm} says that if $X$ is non-spin and $[S] = w_2(X) \in H^2(X ; \mathbb{Z}_2)$, then 
\[
e(S) = \sigma(X) \pm 2 \; ({\rm mod} \; 16).
\]
There are also some known restrictions on the Euler number in the case of definite $4$-manifolds \cite{law}, \cite{lrs}. However, there does not appear to be any known counterpart of the adjunction inequality for non-orientable surfaces. Our next result can be thought of as filling this gap, at least in the case that $[S] = 0 \; ({\rm mod} \; 2)$.

\begin{theorem}\label{thm:1rp2}
Let $X$ be a compact, oriented, smooth $4$-manifold with $H_1(X ; \mathbb{Z}_2) = 0$ and $b_+(X) > 1$. Suppose $X$ is of simple type and let $\mathfrak{s}$ be a spin$^c$-structure such that $SW(X , \mathfrak{s}) \neq 0 \; ({\rm mod} \; 2)$. Let $S \subset X$ be a smoothly embedded real projective plane with $[S] = 0 \in H^2(X ; \mathbb{Z}_2)$. Then $e(S) \le 6$. Furthermore, if $b_+(X) = 3 \; ({\rm mod} \; 4)$ then $e(S) \le 2$.
\end{theorem}

\begin{example}
Let $X_0 = E(n) \# \overline{\mathbb{CP}^2}$ where $n \ge 2$ and $X_1 = (2n-1)\mathbb{CP}^2 \# 10n \overline{\mathbb{CP}^2}$. Then $X_0$ and $X_1$ are homeomorphic. Theorem \ref{thm:1rp2} implies that any smoothly embedded real projective plane $S$ with $[S] = 0 \in H^2(X_0 ; \mathbb{Z}_2)$ satisfies $e(S) \le 6$ (or $e(S) \le 2$ when $n$ is even). On the other hand, for $0 \le k \le 2n-1$ we can find a smoothly embedded real projective plane $S$ in $X_1$ with $[S] = 0$ and $e(S) = 2+4k$. To see this take $S$ to be the connected sum $S_0 \# Q_1 \# \cdots \# Q_k$ where $Q_i$ is a non-singular quadric in the $i$-th copy of $\mathbb{CP}^2$ and $S_0$ is a copy of $\mathbb{RP}^2$ trivially embedded in an open ball and with $e(S_0) = 2$. Note also that the case $k = 2n-1 = b_+(X_1)$ gives $e(S) = 2 + 4 b_+(X)$, which is the maximum possible according to Rokhlin's bound.
\end{example}

\begin{remark}
It is natural to expect an upper bound on $e(S)$ but not a lower bound. One justification for this is that we can blow up $X$ multiple times and attach spheres of self-interesection $-1$ in order to make $e(S)$ as negative as we wish. An interesting question is whether the upper bound $e(S) \le 6$ in Theorem \ref{thm:1rp2} holds without the constraint $[S] = 0$. We have not been able to find any counter-examples to this.
\end{remark}

The next result is an extension to the case of multiple disjoint, embedded real projective planes.

\begin{theorem}
Let $X$ be as in Theorem \ref{thm:1rp2}. Suppose that $S_1, \dots , S_r \subset X$ are disjoint embedded real projective planes whose Poincar\'e dual classes satisfy $[S_1] + \cdots + [S_r] = 0 \in H^2(X ; \mathbb{Z}_2)$, but any $r-1$ of them are linearly independent in $H^2(X ; \mathbb{Z}_2)$. Then $e(S_1) + \cdots + e(S_r) \le \max\{ 0 , 8-2r\}$.
\end{theorem}

\subsection{Seiberg--Witten invariants of branched covers}

We give a formula for the Seiberg--Witten invariants of cyclic branched covers. The general formula is too complicated to give in the introduction, so we will give the $p=2$ case here and give only a rough description of the formula for $p > 2$. Prior to this paper there has been surprisingly little work on computing Seiberg--Witten invariants of branched covers. Ruan--Wang give a formula for branched double covers where the branch locus consists of a single surface of genus $g > 1$ and non-negative self-intersection \cite{ruwa}. Park extends this result to the case of surfaces $g=1$ and zero self-intersection \cite{park}. Covers of degree greater than $2$ and covers with branch locus consisting of multiple surfaces do not seem to have been considered previously.

\begin{theorem}
Let $X$ be a compact, oriented, smooth $4$-manifold with $H_1(X ; \mathbb{Z}_2) = 0$ and $b_+(X) > 1$ and suppose $X$ has simple type. Let $\pi : \widetilde{X} \to X$ be a branched double cover such that $b_1(\widetilde{X}) = 0$ and the branch locus $S_1, \dots , S_r \subset X$ consists of disjoint spheres. Given a spin$^c$-structure $\mathfrak{s}$ on $X$ with $d(X , \mathfrak{s}) = 0$, orient the spheres $S_1, \dots , S_r$ such that $\langle c(\mathfrak{s}) , [S_i] \rangle \ge 0$ for each $i$. We have an associated spin$^c$-structure $\widetilde{\mathfrak{s}}$ on $\widetilde{X}$ such that $c(\widetilde{\mathfrak{s}}) = \pi^*(\mathfrak{s}) + \sum_i [\widetilde{S}_i]$, where $\widetilde{S}_i = \pi^{-1}(S_i)$. Then
\[
d(\widetilde{X} , \widetilde{\mathfrak{s}} ) = r + \frac{1}{2}\left( \sum_i |\langle c(\mathfrak{s}) , [S_i] \rangle | + [S_i]^2 \right)
\]
(where $d(\widetilde{X} , \widetilde{\mathfrak{s}} )$ is the expected dimension of the Seiberg--Witten moduli space) and we have:
\begin{itemize}
\item[(1)]{If $d(\widetilde{X} , \widetilde{\mathfrak{s}}) = 0$, then 
\[
SW(\widetilde{X} , \widetilde{\mathfrak{s}}) = SW(X , \mathfrak{s}) + SW(X , L \otimes \mathfrak{s}) \; ({\rm mod} \; 2)
\]
where $L \in H^2(X ; \mathbb{Z})$ satisfies $[S_1]+ \cdots + [S_r] = 2L$.}
\item[(2)]{If $d(\widetilde{X} , \widetilde{\mathfrak{s}}) = 2$, then
\[
SW(\widetilde{X} , \widetilde{\mathfrak{s}}) = SW(X , \mathfrak{s}) = SW(X , L \otimes \mathfrak{s}) \; ({\rm mod} \; 2).
\]
}
\item[(3)]{In all other cases, $SW(\widetilde{X} , \widetilde{\mathfrak{s}}) = 0 \; ({\rm mod} \; 2)$.}
\end{itemize}

\end{theorem}

Consider now the case of a cyclic branched cover $\pi : \widetilde{X} \to X$ of degree $p$ a prime and with branch locus $S_1, \dots, S_r$ a collection of spheres. Assume that $b_1(\widetilde{X}) = 0$. We introduce the notion of {\em sharp} spin$^c$-structures on $X$, namely those which satisfy $| \langle c(\mathfrak{s}) , [S_i] \rangle | + [S_i]^2 \le 0$. It turns out that the Seiberg--Witten invariants for the sharp spin$^c$-structures completely determines the Seiberg--Witten invariants for all spin$^c$-structures. Moreover when $X$ has simple type only the sharp spin$^c$-structures can have non-vanishing Seiberg--Witten invariants. In Section \ref{sec:spinc} we show how to any $\mathbb{Z}_p$-invariant sharp spin$^c$-structure $\widetilde{\mathfrak{s}}$ on $\widetilde{X}$ we can associate a collection $\{ \mathfrak{s}_j \}_{j=0}^{p-1}$ of sharp spin$^c$-structures on $X$ and we prove an equality of the form
\[
SW(\widetilde{X} , \widetilde{\mathfrak{s}} ) = \sum_{j=0}^{p-1} t_j SW(X , \mathfrak{s}_j) \; ({\rm mod} \; p)
\]
for certain coefficients $t_j \in \mathbb{Z}_p$. The coefficients $t_j$ are given by an explicit, but complicated formula which involves the topology of $X$ and $S_1, \dots , S_r$.

\subsection{Potential constructions of non-simple type $4$-manifolds}

As a consequence of our study of branched coverers, we have discovered a number of potential constructions of $4$-manifolds of non-simple type starting from a $4$-manifold of simple type and a certain configuration of embedded surfaces. Although we have not been successful in finding such a configuration of surfaces, we have also not found any compelling argument to rule out their existence.

\begin{theorem}
Let $X$ be a compact, oriented, smooth $4$-manifold with $H_1(X ; \mathbb{Z}_2) = 0$ and $b_+(X) > 1$. Suppose $X$ has simple type and that $SW(X , \mathfrak{s}) \neq 0 \; ({\rm mod} \; 2)$ for some spin$^c$-structure $\mathfrak{s}$. Suppose $S_1, \dots , S_r \subset X$ are disjoint smoothly embedded spheres whose Poincar\'e dual classes satisfy $[S_1] + \cdots + [S_r] = 0 \in H^2(X ; \mathbb{Z}_2)$ but any $r-1$ of them are linearly independent in $H^2(X ; \mathbb{Z}_2)$. If $\sum_i | \langle c(\mathfrak{s}) , [S_i] \rangle | = 2r-4$ and $\sum_i [S_i]^2 = 8-4r$, then the branched double cover $\widetilde{X} \to X$ of $X$ branched over $S_1, \dots , S_r$ does not have simple type.
\end{theorem}

It should be noted that the conditions in the above theorem can only be met if $r \ge 4$. So we need at least four spheres to obtain a branched double cover of non-simple type. On the other hand the situation for real projective planes is much simpler:

\begin{theorem}
Let $X$ be a compact, oriented, smooth $4$-manifold with $H_1(X ; \mathbb{Z}_2) = 0$ and $b_+(X) > 1$. Suppose $X$ has simple type and that $SW(X , \mathfrak{s}) \neq 0 \; ({\rm mod} \; 2)$ for some spin$^c$-structure $\mathfrak{s}$. 
\begin{itemize}
\item[(1)]{Suppose there exists an embedded real projective plane $S \subset X$ with $[S] = 0 \in H^2(X ; \mathbb{Z}_2)$ and $e(S) = 6$. Then the $4$-manifold $X'$ obtained from $X$ by removing a tubular neighbourhood of $S$ and replacing it with a copy of the $D_8$-plumbing does not have simple type.}
\item[(2)]{Suppose there exists disjoint embedded real projective planed $S_1,S_2 \subset X$ with $[S_1] = [S_2] \in H^2(X ; \mathbb{Z}_2)$, $[S_i] \neq 0$ for $i=1,2$ with $e(S_1) + e(S_2) = 4$ and $e(S_i) \ge -2$. Then the $4$-manifold $\widehat{X}$ obtained from $X$ by taking the double cover of $X$ branched over $S_1,S_2$, removing tubular neighbourhoods of their pre-images and replacing them with suitably chosen negative definite plumbings does not have simple type.
}
\end{itemize}

\end{theorem}

For details on choosing the negative definite plumbings, see Section \ref{sec:rp2}.

Lastly, we consider embedded spheres with a cusp. By this we mean a sphere $C$ and a smooth injective map $C \to X$ which is an embedding except at one point, the cusp, where $C$ is locally given as a cone over the right-handed torus knot $T_{2,3}$ for a right-handed cusp, or as a cone over $-T_{2,3}$ for a left-handed cusp.

\begin{theorem}
Let $X$ be a compact, oriented, smooth $4$-manifold with $b_1(X) = 0$, $b_+(X) > 1$ and suppose $X$ has simple type. Let $C \subseteq X$ be an embedded sphere with a left-handed cusp. Suppose that $SW(X , \mathfrak{s}) \neq 0$ for some spin$^c$-structure $\mathfrak{s}$.
\begin{itemize}
\item[(1)]{If $| \langle c(\mathfrak{s}) , [C] \rangle | + [C]^2 = 0$ and $-7 \le [C]^2 \le -1$, then there exists a $4$-manifold $X'$ which is not of simple type, obtained by removing a tubular neighbouhood of $C$ and replacing it with a suitably chosen negative definite plumbing.}
\item[(2)]{If $[C]^2 = 0$, then there exists a $4$-manifold $X'$ which is not of simple type.}
\end{itemize}
\end{theorem}

In particular, part (2) of this theorem says that an embedded sphere or self-intersection zero with left-handed cusp in a $4$-manifold of simple type and non-zero Seiberg--Witten invariants would give rise to a $4$-manifold of non-simple type. What makes this striking is that spheres of self-intersection zero with a {\em right-handed} cusp exist in simple type $4$-manifolds in abundance, for example a cusp fibre of an elliptic surface. There is a close relationship between spheres with cusps and real projective planes. If $C \subset X$ is an embedded sphere with a cusp, then by removing a neighbourhood of the cusp and attaching a M\"obius band, we obtain an embedded real projective plane $S \subset X$ such that $[S] = [C] \; ({\rm mod} \; 2)$ and $e(S) = 6 + [C]^2$ for a left-handed cusp, $e(S) = -6 + [C]^2$ for a right-handed cusp. In particular, if $C$ has self-intersection zero and a left-handed cusp, then $e(S) = 6$.

\subsection{Method of calculation}

We outline our strategy to compute the Seiberg--Witten invariants of branched covers. The results on configurations of embedded surfaces also follows from variations of this  strategy. Suppose that $\pi : \widetilde{X} \to X$ is a branched cover with branch locus $S_1, \dots , S_r$ a union of embedded spheres with negative self-intersection. Let $X_0$ be the $4$-manifold with boundary obtained from $X$ by removing tubular neighbourhoods of $S_1, \dots , S_r$. Similarly, define $\widetilde{X}_0$ so that we have a commutative diagram
\[
\xymatrix{
\widetilde{X}_0 \ar[d]^-\pi \ar[r] & \widetilde{X} \ar[d]^-\pi \\
X_0 \ar[r] & X
}
\]
Moreover, $\widetilde{X}_0 \to X_0$ is an {\em unbranched} cover of $4$-manifolds with boundary. Now since $X_0$ and $\widetilde{X}_0$ are compact $4$-manifolds with boundary, it is possible to define {\em relative Seiberg--Witten invariants} for them using Floer homological methods. In fact, since the boundaries of $X_0$ and $\widetilde{X}_0$ are unions of rational homology $3$-spheres with positive scalar curvature, it turns out that the relative Seiberg--Witten invariants take a particularly simple form. This is studied in Section \ref{sec:swpsc}. Manolescu's gluing theorem for relative Bauer--Furuta invariants \cite{man2} can be applied to relate the Seiberg--Witten invariants of $X$ and $X_0$. Similarly the Seiberg--Witten invariants of $\widetilde{X}$ and $\widetilde{X}_0$ are related by the gluing formula. In recent work \cite{bar}, we showed how equivariant Seiberg--Witten theory can be used to compute the Seiberg--Witten invariants of unbranched coverings of closed $4$-manifolds. The proof only used properties of abstract Bauer--Furuta maps and so extends easily to the case of $4$-manifolds with positive scalar curvature boundary. Thus we obtain a formula for the Seiberg--Witten invariants of $\widetilde{X}_0$ is terms those of $X_0$. Together with the gluing formula this gives us a relation between the Seiberg--Witten invariants of $X$ and $\widetilde{X}$.

\subsection{Structure of paper}

A brief outline of the paper is as follows. In Section \ref{sec:spherical} we review some properties of spherical $3$-manifolds, in particular the computation of eta and delta invariants of such spaces. In Section \ref{sec:swpsc} we introduce and study the Seiberg--Witten invariants for $4$-manifolds with positive scalar curvature boundary and prove a gluing formula. We also introduce equivariant Seiberg--Witten invariants for such $4$-manifolds as this is needed for our covering space formula for Sieberg--Witten invariants of covering spaces. In Section \ref{sec:branch} we study branched covers. We find sufficient conditions for a branched cover $\widetilde{X} \to X$ with $b_1(\widetilde{X}) = 0$ to exist and we study the relationship between spin$^c$-structures on $X$ and $\widetilde{X}$. Then in Section \ref{sec:sw} we obtain a formula for the Seiberg--Witten invariants of branched covers. We also obtain a formula for the equivariant Seiberg--Witten invariants of branched covers. In Section \ref{sec:rp2} we consider branched double covers where the branch locus consists of embedded real projective planes and finally in Section \ref{sec:cusp} we study embedded spheres with a cusp singularity.

\noindent{\bf Acknowledgments.} The author was financially supported by an Australian Research Council Future Fellowship, FT230100092.

\section{Spherical $3$-manifolds}\label{sec:spherical}

Let $Y = S^3/G$ where $G$ is a finite subgroup of $SO(4)$ acting freely on $S^3$. Recall that $SO(4)$ double covers $SO(3) \times SO(3)$. Projecting to the two factors we have two homomorphisms $p_1,p_2 : SO(4) \to SO(3)$. Let $H_i = p_i(G)$. Then $H_1,H_2$ are finite subgroups of $SO(3)$. Furthermore it is known that at least one of $H_1,H_2$ is cyclic \cite[\textsection 6.1, Lemma 4]{or}. After possibly reversing the orientation of $Y$, we can assume that $H_1$ is cyclic. Then $G$ is conjugate to a subgroup of $S^1 \times_{\mathbb{Z}_2} SU(2) = U(2)$. Therefore we may assume without loss of generality that $G$ is a subgroup of $U(2)$. Let $W$ denote the closed unit ball in $\mathbb{C}^2$. Then $G$ being a subgroup of $U(2)$ acts on $W$ biholomorphically and isometrically. Furthermore, each $g \in G$ different from the identity acts freely on $S^3$, hence the eigenvalues $\lambda_1(g),\lambda_2(g)$ of $g$ are not equal to $1$. It follows that the origin $0 \in W$ is the only fixed point of $g$ acting on $W$. We will use the $G$-action on $W$ to compute eta invariants of $Y$ via the equivariant APS index theorem \cite{don}.  Since $H^2(W ; \mathbb{R}) = 0$, the equivariant APS index theorem for the signature yields
\[
\eta_{sig}(Y) = \frac{1}{|G|} \sum_{g \in G, \, g \neq 1} \frac{(1 + \lambda_1(g)^{-1})(1+\lambda_2(g)^{-1})}{(1-\lambda_1(g)^{-1})(1-\lambda_2(g)^{-1})}.
\]

Next, we consider the eta invariant for the spin$^c$-Dirac operator. Since $W$ is a complex manifold it has a canonical spin$^c$-structure $\mathfrak{s}_{can}$. The action of $G \subset U(2)$ on $W$ is holomorphic, so has a natural lift to $\mathfrak{s}_{can}$, which makes $\mathfrak{s}_{can}$ into a $G$-equivariant spin$^c$-structure. The K\"ahler structure on $W$ also yields a canonical choice of spin$^c$-connection, which we use to define the spin$^c$-Dirac operator. By restriction, $\mathfrak{s}_{can}$ defines a $G$-equivariant spin$^c$-structure on $S^3$, which descends to a spin$^c$-structure on $Y$ which we will continue to denote by $\mathfrak{s}_{can}$. Any other spin$^c$-structure on $Y$ differs from $\mathfrak{s}_{can}$ by tensoring by a line bundle. Since $b_1(Y) = 0$, all such line bundles are flat and correspond to homomorphisms $\phi : G \to S^1$ via their holonomy. To such a homomorphism $\phi$ we get an equivariant line bundle $\mathbb{C}_\phi$ on $W$, namely $\mathbb{C}_\phi$ is the trivial line bundle $W \times \mathbb{C}$ equipped with the $G$-action $g(w,v) = (gw , \phi(g)v)$. Given a homomorphism $\phi : G \to S^1$ we will denote by $\mathfrak{s}_\phi = \mathbb{C}_\phi \otimes \mathfrak{s}_{can}$ the $G$-equivariant spin$^c$-structure on $W$ obtained by tensoring $\mathfrak{s}_{can}$ by $\mathbb{C}_\phi$. By restriction, $\mathfrak{s}_\phi$ defines a $G$-equivariant spin$^c$-structure on $S^3$ which then descends to a spin$^c$-structure on $Y$ which we continue to denote by $\mathfrak{s}_\phi$. Every spin$^c$-structure on $Y$ arises in this manner from a uniquely determined homomorphism $\phi$. Since $W$ has a $G$-invariant positive scalar curvature metric which restricts to the round metric on $S^3$ (consider $W$ as a hemisphere in $S^4$), the equivariant APS index theorem for the spin$^c$-Dirac operator and a Weitzenbock vanishing argument gives
\[
\eta_{dir}(Y , \mathfrak{s}_\phi) = -\frac{2}{|G|} \sum_{g \in G, \, g \neq 1} \frac{ \phi(g) }{(1 - \lambda_1(g)^{-1})(1-\lambda_2(g)^{-1})},
\]
where $\eta_{dir}(Y , \mathfrak{s}_\phi)$ denotes the eta invariant of the spin$^c$ Dirac operator for $\mathfrak{s}_\phi$ with respect to the metric on $Y$ induced by the round metric on $S^3$.

For spherical $3$-manifolds (or more generally rational homology $3$-spheres with positive scalar curvature), we have (see Section \ref{sec:swf}):
\[
\delta( Y , \mathfrak{s} ) = -n( Y , \mathfrak{s}) = -\frac{1}{2} \eta_{dir}(Y , \mathfrak{s}) + \frac{1}{8} \eta_{sig}(Y).
\]
Using this formula we can in principle compute the delta invariants of any spherical $3$-manifold. 
\[
\delta(Y , \mathfrak{s}_\phi) = \frac{1}{|G|} \sum_{g \in G \, g \neq 1} \frac{ \phi(g) + det(1+g^{-1})/8}{det(1-g^{-1})}.
\]

For computing eta invariants, it will be useful to evaluate the following sums:
\[
\alpha( u , n ) = \sum_{j=1}^{n-1} \frac{ \omega^{ju} }{1-\omega^{-j}}, \quad \beta( u , n ) = \sum_{j=1}^{n-1} \frac{ \omega^{ju} }{(1-\omega^{-j})^2},
\]
where $\omega = e^{2\pi i/n}$.

\begin{lemma}\label{lem:alphabeta}
Let $0 \le u \le n-1$. Then
\[
\alpha(u,n) = \frac{n-1}{2} - u, \quad \beta(u,n) = -\frac{1}{12}(n-1)(n-5) + \frac{u}{2}(n-u-2).
\]
\end{lemma}
\begin{proof}
We have
\[
\alpha(u,n) - \alpha(u-1,n) =  \sum_{j=1}^{n-1} \frac{ \omega^{ju}(1-\omega^{-j})}{1 - \omega^{-j}} = \sum_{j=1}^{n-1} \omega^{ju} = \begin{cases} n-1 & u=0 \; ({\rm mod} \; n), \\ -1 & u \neq 0 \; ({\rm mod} \; n). \end{cases}
\]
So $\alpha(u,n) = \alpha(0,n) - u$ for $0 \le u \le n-1$. Similarly, we find
\[
\beta(u,n) - \beta(u-1,n) = \alpha(u,n)
\]
for all $u$. Hence $\sum_{u=0}^{n-1} \alpha(u,n) = 0$. Substituting $\alpha(u,n) = \alpha(0,n) - u$, we see that $\alpha(0,n) = (n-1)/2$ and so $\alpha(u,n) = (n-1)/2 - u$ for $0 \le u \le n-1$. Then we have
\[
\beta(u,n) - \beta(u-1,n) = \alpha(u,n) = \frac{n-1}{2} - u
\]
for $0 \le u \le n$. Hence
\[
\beta(u,n) = \beta(0,n) + \sum_{k=0}^{u} \left( \frac{n-1}{2} - k \right) = \beta(0,n) + \frac{u}{2}(n-u-2)
\]
for $0 \le u \le n$. If we set $\gamma(u,n) = \sum_{j=1}^{n-1} \omega^{ju}/(1-\omega^{-j})^3$, then $\beta(u,n) = \gamma(u,n) - \gamma(u-1,n)$. This shows that $\sum_{u=0}^{n-1} \beta(u,n) = 0$, hence
\[
n \beta(0,n) + \sum_{u=1}^{n-1} \frac{u}{2}(n-u-2) = 0.
\]
Recall that $\sum_{u=1}^{n-1} u^2 = n(n-1)(2n-1)/6$. It follows that $\beta(0,n) = -\frac{1}{12}(n-1)(n-5)$ and hence $\beta(u,n) = -(n-1)(n-5)/12 + u(n-u-2)/2$.
\end{proof}

\subsection{Lens spaces}

For any $p \ge 1$ and any $q \in \mathbb{Z}$ coprime to $p$, we obtain a finite subgroup $G = \mathbb{Z}_p = \langle \sigma \rangle \subset SO(4)$, where $\sigma( z_1 , z_2 ) = ( \omega z_1 , \omega^q z_2)$, $\omega = e^{2 \pi i/p}$. The resulting spherical $3$-manifold $Y = S^3/G$ is the lens space $L(p,q)$. The eta invariant for the signature operator is then
\[
\eta_{sig}(Y) = \frac{1}{p} \sum_{j=1}^{p-1} \frac{(1+\omega^{-j})(1+\omega^{-qj})}{(1-\omega^{-j})(1-\omega^{-qj})} = -4 s(q,p),
\]
where $s(q,p)$ is the Dedekind sum
\[
s(q,p) = -\frac{1}{4p} \sum_{j=1}^{p-1} \frac{(1+\omega^{-j})(1+\omega^{-qj})}{(1-\omega^{-j})(1-\omega^{-qj})} = -\frac{1}{p} \sum_{j=1}^{p-1} \frac{1}{(1-\omega^{-j})(1-\omega^{-qj})} + \frac{1}{4} - \frac{1}{4p}.
\]
For each $u \in \mathbb{Z}_p$, let $\phi_u : G \to S^1$ be the homomorphism $\phi_u( \sigma ) = \omega^u$. Then we obtain an identification between $\mathbb{Z}_p$ and spin$^c$-structures on $L(p,q)$ where $u \in \mathbb{Z}_p$ corresponds to $\mathfrak{s}_{\phi_u}$. For simplicity, we write $\mathfrak{s}_u$ for $\mathfrak{s}_{\phi_u}$. We have
\[
\eta_{dir}( L(p,q) , \mathfrak{s}_u ) = -\frac{2}{p} \sum_{j=1}^{p-1} \frac{ \omega^{ju}}{(1-\omega^{-j})(1-\omega^{-qj})}
\]
and hence
\[
\delta( L(p,q) , \mathfrak{s}_u ) = \frac{1}{p} \sum_{j=1}^{p-1} \frac{ \omega^{ju} }{(1-\omega^{-j})(1-\omega^{-qj})} - \frac{1}{2} s(q,p).
\]
If $q=1$ and $0 \le u \le p-1$, then one finds
\[
\delta( L(p,1) , \mathfrak{s}_u) = - \frac{ (2u+2-p)^2 }{8p} + \frac{1}{8}.
\]

\subsection{Prism manifolds}\label{sec:prism}

We consider specifically the prism manifolds of the form $Y(n) = S^3/G$, where $G = D^*_n = \langle x , y \; | \;  x^2 = (xy)^2 = y^n \rangle$ is the binary Dihedral group of order $4n$. Here $x,y \in SU(2)$ are given by $x(z_1,z_2) = (-z_2 , z_1)$, $y(z_1,z_2) = (\omega z_1 , \omega^{-1} z_2)$, $\omega = e^{\pi i/n}$. Elements of the form $y^j$, $0 \le j \le 2n-1$ have eigenvalues $\omega^j , \omega^{-j}$, where $\omega = e^{\pi i /n}$. Elements of the form $x y^j$, $0 \le j \le 2n-1$ have eigenvalues $i,-i$. Hence the eta invariant for the signature operator is
\begin{align*}
\eta_{sig}( Y(n) ) &= \frac{1}{4n}\left( \sum_{j=1}^{2n-1} \frac{(1+\omega^{-j})(1+\omega^j)}{(1-\omega^{-j})(1-\omega^j)} + 2n \frac{(1+i)(1-i)}{(1-i)(1+i)} \right) \\
&= 2s(1,2n) + \frac{1}{2} \\
&= \frac{(2n-1)(n-1)}{6n} + \frac{1}{2} \\
&= \frac{1}{6n}(2n^2 +1).
\end{align*}

To determine the spin$^c$-structures we need to consider homomorphisms $\phi : D^*_n \to S^1$. Since $xyx^{-1} = y^{-1}$, we have $\phi(y) = \pm 1$. We write $\phi(y) = (-1)^v$ where $v \in \{ 0,1\}$. Since $x^2 = y^n$, we have $\phi(x)^2 = (-1)^{nv}$. Then we can write $\phi(x) = i^{ev} (-1)^u$ where $u \in \{0,1\}$ and $e = 0$ or $1$ according to whether $n$ is even or odd. So there are four spin$^c$-structures $\mathfrak{s}_{u,v} = \mathfrak{s}_\phi$ determined by the possible values of $u,v \in \{0,1\}$. The eta invariant for the spin$^c$-Dirac operator is
\begin{align*}
\eta_{dir}( Y(n) , \mathfrak{s}_{u,v} ) &= -\frac{2}{4n}\left( \sum_{j=1}^{2n-1} \frac{ (-1)^{jv} }{(1-\omega^{-j})(1-\omega^j)} + n \frac{ i^{ve}( (-1)^u + (-1)^{u+v} )}{(1-i)(1+i)} \right) \\
&= \frac{1}{2n} \sum_{j=1}^{2n-1} \frac{ \omega^{(nv-1)j} }{(1-\omega^{-j})^2} - \frac{1}{4} i^{ve}( (-1)^u + (-1)^{u+v} ). 
\end{align*}
For $v=0$, we get
\begin{align*}
\eta_{dir}(Y(n) , \mathfrak{s}_{u,0}) &= \frac{1}{2n} \beta(2n-1,2n) - \frac{1}{2}(-1)^u \\
&= \frac{1}{2n}\left( -\frac{(2n-1)(2n-5)}{12} - \frac{2n-1}{2} \right) - \frac{1}{2}(-1)^u \\
&= \frac{1}{24n}\left( -4n^2 + 1 \right) - \frac{1}{2}(-1)^u.
\end{align*}

For $v = 1$, we get
\begin{align*}
\eta_{dir}(Y(n) , \mathfrak{s}_{u,v}) &= \frac{1}{2n} \beta(n-1,2n) \\
&= \frac{1}{2n} \left( \frac{-(2n-1)(2n-5)}{12} + \frac{(n-1)}{2}( n - 1) \right) \\
&= \frac{1}{24n} \left( 2n^2 + 1 \right). 
\end{align*}

Hence
\begin{align*}
\delta( Y(n) , \mathfrak{s}_{u,0} ) &= -\frac{1}{48n}( -4n^2 + 1) + \frac{1}{4}(-1)^u + \frac{1}{48n}(2n^2 + 1) \\
&= \frac{n}{8} + \frac{1}{4}(-1)^u = \frac{ (n+ 2(-1)^u)}{8}
\end{align*}
for $v=0$ and
\begin{align*}
\delta( Y(n) , \mathfrak{s}_{u,1}) &= -\frac{1}{48n}( 2n^2 + 1 ) + \frac{1}{48n}(2n^2 + 1) \\
&=0
\end{align*}
for $v=1$.

If $d$ divides $n$ then there is a covering $\pi : Y(n/d) \to Y(n)$ corresponding to the subgroup of $\pi_1(Y(n)) = D_n^*$ generated by $x$ and $y^d$.

\begin{lemma}\label{lem:pullbackyn}
We have that $\pi^*( \mathfrak{s}_{u,v} ) = \mathfrak{s}_{u,0}$ if $n$ and $d$ are both even and $\pi^*(\mathfrak{s}_{u,v}) = \mathfrak{s}_{u,v}$ otherwise.
\end{lemma}
\begin{proof}
Recall that $Y(n) = S^3/G$ where $G = D_n^*$ and the covering $Y(n/d) \to Y(n)$ corresponds to the subgroup $D_{n/d}^*$ generated by $x$ and $y^d$. Recall that $\mathfrak{s}_{can} = \mathfrak{s}_{0,0}$ is the spin$^c$-structure on $Y(n)$ obtained by descending the canonical spin$^c$-structure on the $4$-ball $W$ using the natural lift of $G$ to the canonical spin$^c$-structure. It follows that $\mathfrak{s}_{can}$ on $Y(n)$ pulls back to $\mathfrak{s}_{can}$ on $Y(n/d)$. Then more generally, for a homomorphism $\phi : D_n^* \to S^1$, we have $\pi^*(\mathfrak{s}_{\phi}) = \mathfrak{s}_{\phi'}$, where $\phi'$ is the restriction of $\phi$ to $D^*_{n/d}$. If $\mathfrak{s}_{\phi} = \mathfrak{s}_{u,v}$ then $\phi$ is given by $\phi(y) = (-1)^v$ and $\phi(x) = i^{ev} (-1)^u$, where $e = 0$ if $n$ is even and $1$ if $n$ is odd. Then $\phi'$ is given by $\phi'(y^d) = (-1)^{vd}$, $\phi'(x) = i^{ev} (-1)^{u}$. If $d$ is odd, then $(-1)^{vd} = (-1)^v$ and it follows easily that $\pi^*( \mathfrak{s}_{u,v}) = \mathfrak{s}_{u,v}$. If $d$ is even, then so is $n$ and we have $\phi'(y^d) = 1$, $\phi'(x) = (-1)^u$. Hence $\pi^*(\mathfrak{s}_{u,v}) = \mathfrak{s}_{u,0}$.
\end{proof}

\subsection{Plumbing on $ADE$ diagrams}

Consider a Dynkin diagram $L$ of type $A,D$ or $E$. Associated to such a diagram is a negative definite plumbing $P(L)$ whose plumbing graph is the Dynkin diagram $L$ and where every node has weight $-2$. The boundary of $P(L)$ is a Seifert $3$-manifold which is a spherical $3$-manifold.

For the diagram $A_{n-1}$ ($n \ge 2$), the boundary is the lens space $L(n,n-1) = -L(n,1)$. For the diagram $D_{n+2}$ ($n \ge 1$), the boundary is the Prism manifold $Y(n) = M(1 ; (2,1) , (2,1) , (n,1))$, where $M(e ; (a_1,b_1) , \dots , (a_n,b_n))$ denotes the Seifert fibration over $S^2$ with Euler class $e$ and $n$ singular fibres of types $(a_1,b_1), \dots , (a_n,b_n)$. To see that the orientation on $Y(n)$ as the boundary of $P(D_{n+2})$ agees with the orientation used in Section \ref{sec:prism}, we compare delta invariants. As shown in Section \ref{sec:prism}, the delta invariants of $Y(n)$ are $(n+2)/8, (n-2)/8, 0 , 0$. On the other hand, the spin structure on $P(D_{n+2})$ with zero characteristic induces a spin$^c$-structure on its boundary $\partial P(D_{n+2})$ whose delta invariant equals $-\sigma(P)/8 = (n+2)/8$. Since $(n+2)/8$ occurs as a delta invariant of $Y(n)$, but not as a delta invariant of $-Y(n)$, we have $\partial P(D_{n+2}) = Y(n)$, as claimed.

For the diagrams $E_6,E_7,E_8$, we get the spherical $3$-manifolds $S^3/G$ where $G = T^*, O^*$ or $I^*$ is one of the binary tetrahedral, binary octahedral or binary icosahedral groups. In terms of Seifert structures, we have $\partial P(E_6) = M(1 ; (2,1) , (3,1) , (3,1))$, $\partial P(E_7) = M(1 ; (2,1) , (3,1) , (4,1))$ and $\partial P(E_8) = M(1 ; (2,1) , (3,1) , (5,1) )$, the Poincar\'e homology sphere. To see that the two orientations agree (one as a quotient of $S^3$, the other as the boundary of a plumbing) is a simple matter of computing delta invariants. We omit the details.

\subsection{Circle bundles over $\mathbb{RP}^2$}

For each $n \in \mathbb{Z}$, let $R(n)$ be the unique closed disc bundle over $\mathbb{RP}^2$ whose fibre orientation class equals $w_1(\mathbb{RP}^2)$ (so that the total space of $R(n)$ is oriented) and with Euler class equal to $n$ under the isomorphism $H^2( \mathbb{RP}^2 ; \mathbb{Z}_{w_1} ) \cong \mathbb{Z}$, where $\mathbb{Z}_{w_1}$ denotes the locally constant sheaf with structure group $\mathbb{Z}$ and with monodromy equal to $w_1(\mathbb{RP}^2) \in H^1( \mathbb{RP}^2 ; \mathbb{Z}_2) = Hom( \pi_1(\mathbb{RP}^2) , \mathbb{Z}_2)$. Alternatively, $n$ is the self-intersection number of the zero section of $R(n)$. The boundary of $R(n)$ is a circle bundle over $\mathbb{RP}^2$.

For $n > 0$, we claim that the boundary of $R(n)$ is the prism manifold $Y(n)$. Ignoring orientations, it is clear that there is a homeomorphism between $\partial R(n)$ and $Y(n)$ by comparing fundamental groups. To see that the orientations agree, we observe that $R(n)$ is double covered by $X(2n)$, the closed disc bundle over $S^2$ with Euler class $2n$. Hence $\partial R(n)$ is double covered by $\partial X(2n) = -L(2n,1)$. On the other hand the description of $Y(n)$ in Section \ref{sec:prism} makes it clear that $Y(n)$ is also double covered by $-L(2n,1)$. Hence $\partial R(n) = Y(n)$. Recall that $Y(n)$ has four spin$^c$-structures $\mathfrak{s}_{u,v}$ and the ones with $v=1$ have delta invariant $0$. $R(n)$ has two spin$^c$-structures and their restrictions to $Y(n)$ are $\mathfrak{s}_{0,1}, \mathfrak{s}_{1,1}$.

Since $Y(n) = \partial R(n)$ for $n > 0$, we will extend the definition of $Y(n)$ to all $n \in \mathbb{Z}$ by declaring that $Y(n) = \partial R(n)$ for all $n$. This gives $Y(-n) = -Y(n)$. The case $Y(0)$ is particularly interesting. In this case $Y(0)$ is $\mathbb{RP}^3 \# \mathbb{RP}^3 = L(2,1) \# L(2,1)$. Since $\mathbb{RP}^3$ has two spin$^c$-structures $\mathfrak{s}^{\pm}$ with delta invariants $\delta( \mathbb{RP}^3 , \mathfrak{s}^{\pm} ) = \pm 1/8$, we see that $Y(0)$ has four spin$^c$-structures $\mathfrak{s}^{\pm} \# \mathfrak{s}^{\pm}$ whose delta invariants are $1/4, -1/4, 0 , 0$. Notice that this fits nicely with the calculation in Section \ref{sec:prism} where we showed that for $n>0$, the delta invariants for $Y(n)$ are $(n+2)/8, (n-2)/8 , 0 , 0$. Setting $n=0$ in these expressions gives $1/4, -1/4 , 0 , 0$ (of course we can not simply set $n=0$ without justification because $Y(0)$ is not a spherical manifold. It has infinite fundamental group).

We have established above that $Y(n)$ for $n>0$ bounds a negative definite plumbing (the $D_{n+2}$ plumbing). We also have that $Y(0) = \mathbb{RP}^3 \# \mathbb{RP}^3$ bounds a negative definite $4$-manifold, the boundary connected sum of two copies of $X(-2)$. Lastly, since $D_3 = A_3$, it follows that $Y(1) = -L(4,1)$ and hence $Y(-1) = L(4,1)$ is the boundary of the negative definite plumbing $X(-4)$ (in fact we also have that $L(4,1)$ bounds a rational homology $4$-ball).

\subsection{Seiberg--Witten Floer homotopy type}\label{sec:swf}

Let $Y$ be a rational homology $3$-sphere and let $\mathfrak{s}$ be a spin$^c$-structure on $Y$. Manolescu constructed an $S^1$-equivariant stable homotopy type $SWF(Y,\mathfrak{s})$ using Conley index theory \cite{man1}. $SWF(Y , \mathfrak{s})$ belongs to a suitably defined equivariant Spanier--Whitehead category. In the case that $Y$ admits a metric $g$ of positive scalar curvature, the construction vastly simplifies. Let $V$ denote the global Coulomb slice \cite[\textsection 3]{man1} for $(Y,\mathfrak{s},g)$ and let $V^\mu_\lambda$ denote its finite-dimensional approximation, as in \cite[\textsection 4]{man1}. Using the same argument given in \cite[\textsection 10]{man1}, we see that in the presence of a positive scalar curvature metric $g$, the Conley index can be taken to be $I = (V^0_\lambda)$ and then $SWF(Y,\mathfrak{s},g) = \Sigma^{-V^0_\lambda} I \cong S^0$. Recall that the metric indendent Seiberg--Witten Floer homotopy type $SW(Y,\mathfrak{s})$ is given by a fractional de-suspension of $SWF(Y , \mathfrak{s} , g)$:
\[
SWF(Y , \mathfrak{s}) = \Sigma^{-n(Y,\mathfrak{s},g)\mathbb{C}} SWF(Y , \mathfrak{s} , g ),
\]
where $n(Y , \mathfrak{s},g) \in \mathbb{Q}$ is defined as
\[
n(Y , \mathfrak{s},g) = \frac{1}{2}\eta_{dir}(Y , \mathfrak{s},g) - \frac{1}{2}k(Y,\mathfrak{s},g) - \frac{1}{8}\eta_{sig}(Y,g)
\]
where $k(Y , \mathfrak{s},g)$ is the complex dimension of the kernel of the spin$^c$-Dirac operator for $(Y,\mathfrak{s},g)$. In the positive scalar curvature case this reduces to
\[
n(Y , \mathfrak{s},g) = \frac{1}{2}\eta_{dir}(Y , \mathfrak{s},g) - \frac{1}{8}\eta_{sig}(Y,g).
\]
Note that since $SW(Y , \mathfrak{g}) = \Sigma^{-n(Y,\mathfrak{s},g)\mathbb{C}} SW(Y , \mathfrak{s} , g ) \cong \Sigma^{-n(Y,\mathfrak{s},g)\mathbb{C}} S^0$, the Fr\o yshov invariant $\delta(Y , \mathfrak{s})$ (equal to one half of the Ozsv\'ath--Szab\'o $d$-invariant $d(Y,\mathfrak{s})$) is given by
\[
\delta(Y, \mathfrak{s}) = -n(Y , \mathfrak{s} , g ) =  -\frac{1}{2}\eta_{dir}(Y , \mathfrak{s},g) + \frac{1}{8}\eta_{sig}(Y,g)
\]
for any choice of positive scalar curvature metric $g$.

Suppose that $X$ is a compact, oriented, smooth $4$-manifold with boundary $Y = \partial X$ a rational homology $3$-sphere with positive scalar curvature. Choose a metric $g$ on $X$ which is cylindrical near the boundary and whose restriction to $Y$ has positive scalar curvature. Let $\mathfrak{s}$ be a spin$^c$-structure on $X$ and choose a spin$^c$-connection $A$ such that the curvature $F_A$ vanishes in a neighbourhood of $Y$. Then following \cite[\textsection 6]{man1}, the $APS$ index $ind_{APS}(X , \mathfrak{s} , A ,  g)$ of the spin$^c$-Dirac operator determined by $(X,\mathfrak{s},A,g)$ is given by:
\begin{align*}
ind_{APS}(X , \mathfrak{s} , A , g) &= \delta(X , \mathfrak{s}) + n(Y , \mathfrak{s} , g) \\
&= \delta(X , \mathfrak{s}) - \delta(Y , \mathfrak{s}),
\end{align*}
where we set
\[
\delta(X , \mathfrak{s}) = \frac{ c(\mathfrak{s})^2 - \sigma(X) }{8}.
\]
Interestingly, this shows that under the stated conditions on $A,g$, $ind_{APS}(X,\mathfrak{s},A,g)$ depends only $X,\mathfrak{s}$ and not on $A$ or $g$. As such we will often denote the APS index as $ind_{APS}(X,\mathfrak{s})$.

If $X$ has multiple boundary components, each of which is a rational homology $3$-sphere with positive scalar curvature then a similar result holds.

\section{Seiberg--Witten invariants for $4$-manifolds with psc boundary}\label{sec:swpsc}

Let $X$ be a compact, oriented, smooth $4$-manifold with boundary a disjoint union of rational homology spheres which admit metrics of positive scalar curvature. In such a case we will simply say that $X$ has psc boundary. For simplicity we will also assume $b_1(X) = 0$. We will always use metrics which are cylindrical near the boundary and whose restriction to each boundary component has positive scalar curvature. Let $\mathfrak{s}$ be a spin$^c$-structure on $X$. Using the results of Section \ref{sec:swf}, it follows that the relative Bauer--Furuta map of $(X , \mathfrak{s})$ in the sense of \cite{man1} can be represented by an $S^1$-equivariant map
\[
f : V^+ \wedge U^+ \to (V')^+ \wedge (U')^+,
\]
where $V,V'$ are complex vector spaces of ranks $a,a'$, $U,U'$ are real vector spaces of ranks $b,b'$, the superscript $+$ denotes one-point compactification. The $S^1$-action is trivial on $U,U'$ and acts by scalar multiplication on $V,V'$. The ranks $a,a',b,b'$ are such that $b'-b = b_+(X)$, $a-a' = ind_{APS}(X , \mathfrak{s})$. Furthermore, the restriction of $f$ to the $S^1$-fixed point set can be assumed to be given by a linear inclusion of vector spaces $U \to U'$. Thus $f$ is an abstract monopole map in the sense of \cite[Definition 2.1]{bar}. We define the expected dimension of the Seiberg--Witten moduli space to be
\begin{align*}
d(X , \mathfrak{s}) &= 2ind_{APS}(X,\mathfrak{s}) - b_+(X) - 1 \\
& =2 \delta(X , \mathfrak{s}) + 2 n(Y , \mathfrak{s}|_Y) - b_+(X) - 1 \\
& = 2 \delta(X , \mathfrak{s}) - 2 \delta(Y , \mathfrak{s}|_Y) - b_+(X) - 1
\end{align*}
where $n(Y, \mathfrak{s}|_Y)$ means that we sum up $n(Y_i , \mathfrak{s}|_{Y_i})$ over all boundary components (weighting by $-1$ for ingoing components) and $\delta(Y , \mathfrak{s}|_Y)$ is defined similarly. We are using that the boundary has positive scalar curvature to deduce that $n(Y , \mathfrak{s}|_Y) = -\delta(Y , \mathfrak{s}|_Y)$. The meaning of $d(X,\mathfrak{s})$ is that it is the dimension of $f^{-1}(\varphi)/S^1$ if $\varphi$ is an $S^1$-invariant regular value of $f$.

Assume $b_+(X)>1$. We then define the Seiberg--Witten invariant of $(X, \mathfrak{s})$ to be the abstract Seiberg--Witten invariant of its relative Bauer--Furuta map in the sense of \cite[Definition 2.3]{bar}. This takes the form of a map
\[
SW_{X , \mathfrak{s}} : H^*_{S^1}(pt ; \mathbb{Z}) \to \mathbb{Z}.
\]
Recall that $H^*_{S^1}(pt ; \mathbb{Z}) \cong \mathbb{Z}[x]$ where $x$ has degree $2$. Moreover, $SW_{X , \mathfrak{s}}( x^m ) = 0$ unless $d(X , \mathfrak{s}) = 2m$. We also define $SW(X,\mathfrak{s}) = SW_{X , \mathfrak{s}}( x^m)$ if $d(X,\mathfrak{s}) = 2m$ is positive and even and $SW(X ,\mathfrak{s}) = 0$ otherwise. We will also refer to $SW(X , \mathfrak{s})$ as the Seiberg--Witten invariant of $(X,\mathfrak{s})$. If $b_+(X) = 1$, then we again have Seiberg--Witten invariants, but now depending on a chamber. If $X$ has no boundary, then the relative Bauer--Furuta invariant of $(X , \mathfrak{s})$ coincides with the ordinary (non-relative) Bauer--Furuta invariant. Hence in such cases $SW(X,\mathfrak{s})$ coincides with the usual Seiberg--Witten invariant of $(X,\mathfrak{s})$ (by \cite[Theorem 2.24]{bako}).

The Seiberg--Witten invariants of $X$ do not depend on the choice of positive scalar curvature metric on the boundary. This follows because the space of such metrics on $\partial X$ is path-connected (in fact contractible) \cite[Theorem 1.1]{bakl}.

Suppose that $X_1,X_2$ are two connected $4$-manifolds with psc boundaries and that we glue $X_1$ and $X_2$ along a common boundary component, say $X = X_1 \cup_{Y} X_2$, where $Y$ is an outgoing component of $X_1$ and an ingoing component of $X_2$. The resulting $4$-manifold $X$ might depend on the choice of diffeomorphism used to identify the boundary components, but the gluing formula below implies that the Seiberg--Witten invariants of $X$ will not depend on this choice. Let $\mathfrak{s}_1, \mathfrak{s}_2$ be spin$^c$-structures on $X_1,X_2$ whose restrictions to $Y$ agree. Then (since $H^1(Y ; \mathbb{Z}) = 0$) there is a unique way to glue $\mathfrak{s}_1$ and $\mathfrak{s}_2$ together to form a spin$^c$-structure $\mathfrak{s} = \mathfrak{s}_1 \cup_Y \mathfrak{s}_2$ on $X$. Manolescu's gluing formula \cite{man2} implies that $f = f_1 \wedge f_2$, where $f_1,f_2$ are the relative Bauer--Furuta invariants for $(X_1 , \mathfrak{s}_1)$, $(X_2 , \mathfrak{s}_2)$ and $f$ is the relative Bauer--Furuta invariant of $(X , \mathfrak{s})$.

The abstract gluing formula for monopole maps \cite[Theorem 8.2]{bar} implies the following: 
\begin{theorem}\label{thm:pscglue}
Let $X = X_1 \cup_Y X_2$ be two $4$-manifolds with psc boundaries glued together along a common component. Assume that $b_1(X) = 0$ and $b_+(X) > 1$. If $b_+(X_1)$ and $b_+(X_2)$ are both positive, then $SW(X , \mathfrak{s}) = 0$. If $b_+(X_1) > 1$ and $b_+(X_2) = 0$, then
\[
SW_{X,\mathfrak{s}}( x^m ) = SW_{X_1 , \mathfrak{s}_1}( x^{m - ind_{APS}( X_2 , \mathfrak{s_2})} ).
\]
\end{theorem}

Note also that
\[
d(X , \mathfrak{s})  = d(X_1 , \mathfrak{s}_1) + 2 ind_{APS}(X_2 , \mathfrak{s}_2), \quad d(X,\mathfrak{s}) = d(X_1 , \mathfrak{s}_1) \; ({\rm mod} \; 2).
\]
Therefore, we deduce from Theorem \ref{thm:pscglue} that:
\[
SW(X , \mathfrak{s}) = \begin{cases} SW(X_1 , \mathfrak{s}_1) & \text{if } d(X,\mathfrak{s}) \ge 0, \\ 0 & \text{otherwise}. \end{cases}
\]
By repeatedly applying Theorem \ref{thm:pscglue}, we obtain a similar formula for $SW(X,\mathfrak{s})$ when $X$ is obtained from $X_1$ by attaching negative definite $4$-manifolds to some subset of the boundary components. 

\begin{definition}\label{def:sharp}
Let $W$ be a compact, oriented, smooth, negative definite $4$-manifold with psc boundary $Y$. Let $\mathfrak{s}$ be a spin$^c$-structure on $W$. We will say $\mathfrak{s}$ is {\em sharp} if $\delta(Y , \mathfrak{s}|_Y) = \delta(W , \mathfrak{s})$ (note that the Fr\o yshov inequality gives $\delta(Y , \mathfrak{s}|_Y) \ge \delta(W , \mathfrak{s})$ for any spin$^c$-structure on $W$).
\end{definition}

\begin{proposition}\label{prop:indle}
Let $W$ be a compact, oriented, smooth, negative definite $4$-manifold with psc boundary $Y$. Let $\mathfrak{s}$ be a spin$^c$-structure on $W$. Then $ind_{APS}( W , \mathfrak{s} ) \le 0$ with equality if and only if $\mathfrak{s}$ is sharp.
\end{proposition}
\begin{proof}
Immediate from $ind_{APS}(W , \mathfrak{s}) = \delta(W , \mathfrak{s}) - \delta(Y , \mathfrak{s}|_Y)$ and the Fr\o yshov inequality $\delta(Y , \mathfrak{s}|_Y) \ge \delta(W , \mathfrak{s})$.
\end{proof}

\begin{proposition}
Let $W$ be a compact, oriented, smooth $4$-manifold with psc boundary and with $b_2(W) = 0$. Then every spin$^c$-structure on $W$ is sharp. Furthermore, $\delta(Y , \mathfrak{s}) = 0$ for every spin$^c$-structure on $Y$ that extends to a spin$^c$-structure on $W$.
\end{proposition}
\begin{proof}
Let $\mathfrak{s}$ be a spin$^c$-structure on $W$. Applying Proposition \ref{prop:indle} to both orientations on $W$ gives $ind_{APS}(W , \mathfrak{s}) = 0$. Furthermore, $0 = ind_{APS}(W , \mathfrak{s}) = \delta(W , \mathfrak{s}) - \delta(Y , \mathfrak{s}|_Y)$. But $b_2(W) = 0$ implies that $\delta(W , \mathfrak{s}) = 0$ and hence $\delta(Y , \mathfrak{s}|_Y) = 0$.
\end{proof}

\begin{proposition}
Suppose that $W$ is a negative definite plumbing on an almost rational graph (\cite[Definition 8.1]{nem}) and suppose that the boundary $Y$ of $W$ is a rational homology $3$-sphere which admits a positive scalar curvature metric. Then every spin$^c$-structure on $Y$ admits an extension to a sharp spin$^c$-structure on $W$.
\end{proposition}
\begin{proof}
Immediate from \cite[Theorem 8.3]{nem} and the definition of sharpness.
\end{proof}

\subsection{Equivariant and reduced Seiberg--Witten invariants}\label{sec:eqred}

Let $p$ be a prime and $G = \mathbb{Z}_p = \langle \sigma \rangle$ a cyclic group of prime order. Suppose that $X$ is a closed, oriented smooth $4$-manifold with $b_1(X) = 0$. Suppose that $G$ acts on $X$ by orientation preserving diffeomorphism and that $dim( H^+(X)^G ) > 1$. Let $\mathfrak{s}$ be a $G$-invariant spin$^c$-structure. Then the {\em $G$-equivariant Seiberg--Witten invariant} of $(X , \mathfrak{s})$ takes the form of a map of $H^*_{\mathbb{Z}_p}(pt)$-modules
\[
SW_{G , X , \mathfrak{s}} : H^*_{G_{\mathfrak{s}}}(pt) \to H^{* - d(X,\mathfrak{s})}_{\mathbb{Z}_p}(pt)
\]
where $G_{\mathfrak{s}}$ is a certain $S^1$-central extension of $G$ and where the cohomology groups will be taken with $\mathbb{Z}_p$-coefficients. In the case $G = \mathbb{Z}_p$, the central extension $G_{\mathfrak{s}}$ is trivial and a choice of splitting $G_{\mathfrak{s}} \cong S^1 \times \mathbb{Z}_p$ is equivalent to a choice of lift of the $G$-action to $\mathfrak{s}$. Fixing such a choice we have a corresponding isomorphism $H^*_{G_{\mathfrak{s}}}(pt) \cong H^*_{\mathbb{Z}_p}(pt)[x]$, where $deg(x) = 2$ and then $SW_{G , X , \mathfrak{s}}$ is equivalent to giving a series of cohomology classes
\[
SW_{G , X , \mathfrak{s}}(x^m) \in H^{2m-d(X,\mathfrak{s})}_{\mathbb{Z}_p}(pt).
\]
Under the forgetful map $H^*_{\mathbb{Z}_p} \to H^*(pt) = \mathbb{Z}_p$, $SW_{G , X , \mathfrak{s}}(x^m)$ reduces to $SW_{X , \mathfrak{s}}(x^m)$.

In addition to the equivariant Seiberg--Witten invariants, we have another series of invariants which in \cite{bar} are referred to as the {\em reduced Seiberg--Witten invariants} $\overline{SW}^s_{G , X,\mathfrak{s}}$. Let $(X , \mathfrak{s})$ be as above and let $s : G \to G_{\mathfrak{s}}$ be a choice of splitting of the extension $G_{\mathfrak{s}}$. Then the {\em reduced Seiberg--Witten invariant} of $(X , \mathfrak{s})$ with respect to the splitting $s$ takes the form of a map
\[
\overline{SW}^s_{G , X , \mathfrak{s}} : H^*_{S^1}(pt) \to \mathbb{Z}.
\]
Let $d(X , \mathfrak{s})^{sG}$ denote the expected dimension of the moduli space of $sG$-invariant solutions to the Seiberg--Witten equations, namely
\[
d(X , \mathfrak{s})^{sG} = 2 ind(X , \mathfrak{s})^G - b_+(X)^G - 1
\]
where $ind(X , \mathfrak{s})^{sG}, b_+(X)^G$ denote the dimensions of the $G$-invariant parts of $ind(X,\mathfrak{s})$ (the index of the spin$^c$-Dirac operator and where $G$ acts via the splitting $s$) and $H^+(X)$. Then $\overline{SW}^s_{G , X , \mathfrak{s}}(x^m) = 0$ except when $d(X , \mathfrak{s})^{sG} = 2m$, where $m \ge 0$. Similar to the ordinary Seiberg--Witten invariant we define $\overline{SW}^s_G(X , \mathfrak{s}) = \overline{SW}^s_{G , X , \mathfrak{s}}(x^m)$ if $d(X,\mathfrak{s})^{sG} = 2m$ is even and non-negative and $\overline{SW}^s_G(X , \mathfrak{s}) = 0$ otherwise.

Supppose now that $X$ is a compact, oriented, smooth $4$-manifold with $b_1(X) = 0$ and psc boundary. Suppose that $G = \mathbb{Z}_p$ acts smoothly and orientation preservingly on $X$. Assume that $G$ sends each component of $\partial X$ to itself. Assume further that the $G$-action on each component of $\partial X$ preserves a positive scalar curvature metric. Then it is possible to choose a $G$-invariant cylindrical end metric on $X$ which has positive scalar curvature on the boundary (choose a cylindrical end metric which near the boundary is $G$-invariant and has positive scalar curvature and then average the metric over $G$). Then we can construct the $G$-equivariant relative Bauer--Furuta invariant of $(X , \mathfrak{s})$ as done in \cite[\textsection 4]{bh}. Since this is an abstract $G$-monopole map in the sense of \cite[Definition 2.1]{bar}, we can define equivariant and reduced Seiberg--Witten invariants of $(X , \mathfrak{s})$ in exactly the same way as in the case where $X$ is closed. The main difference is that the index $ind(X , \mathfrak{s})$ of the spin$^c$-Dirac operator must now be replaced with the corresponding $APS$ index $ind_{APS}(X , \mathfrak{s})$.

Let $X$ be a compact, oriented, smooth $4$-manifold with psc boundary and $b_1(X) = 0$. Let $G = \mathbb{Z}_p$ act on $X$ by orientation preserving diffeomorphisms which send each component of $\partial X$ to itself and which preserves a psc metric on the boundary. Assume that $b_+(X)^G > 1$. Let $\mathfrak{s}$ be a $G$-invariant spin$^c$-structure. Choose a splitting $s : G \to G_{\mathfrak{s}}$. Then for each $j \in \mathbb{Z}_p$, we obtain a new splittting $s_j : G \to G_{\mathfrak{s}}$ by setting $s_j(\sigma) = \omega^{-j} s(\sigma)$, where $\omega = e^{2\pi i/p}$. Any splitting $G \to G_{\mathfrak{s}}$ is given by $s_j$ for some $j$. The splitting $s = s_0$ determines a lift of $G$ to the spinor bundles and hence makes $ind_{APS}(X , \mathfrak{s}_0)$ into a virtual representation of $G$. Define $d_0, \dots , d_{p-1} \in \mathbb{Z}$ by $ind_{APS}(X , \mathfrak{s}_0) = \bigoplus_{j=0}^{p-1} \mathbb{C}_j^{d_j}$, where $\mathbb{C}_j$ is the $1$-dimensional representation of $G$ for which $\sigma$ acts as multiplication by $\omega^j$. Let us also set $b_0 = b_+(X)^G$.

For odd $p$ we have $H^*_{\mathbb{Z}_p} \cong \mathbb{Z}_p[u,v]/(u^2)$, where $deg(u) = 1$, $deg(v)=2$. For $p=2$ we have $H^*_{\mathbb{Z}_2} \cong \mathbb{Z}_2[u]$, where $deg(u) = 1$. In this case we set $v = u^2$ and sometimes write $u = v^{1/2}$. Given integers $n,n_0, \dots , n_{p-1}$ and $j \in \mathbb{Z}_p$, define $c_j(n ; n_0 , \dots , n_{p-1}) \in H^*_{\mathbb{Z}_p}$ by
\[
c_j(n ; n_0 , \dots , n_{p-1}) = \left( \sum_{k_i} \prod_{i | i \neq j} \binom{n_i}{k_i} (i-j)^{n_i-k_i} \right) v^{\sum_i n_i - n}
\]
where the sum is over non-negative integers $k_0 , \dots , \hat{k}_j , \dots , k_{p-1}$ such that $k_0 + \cdots + k_{p-1} = n-n_j$. We extend the definition of $c_j$ to non-integer values of $n$ by taking $c_j(n ; n_0 , \dots , n_{p-1}) = 0$ if $n$ is non-integral.

Let $(X,\mathfrak{s})$ be as above. The following result results relate the equivariant and reduced Seiberg--Witten invariants of $(X, \mathfrak{s})$:

\begin{theorem}\label{thm:eqre}
For any $m_0, \dots , m_p \ge 0$, we have
\begin{align*}
& SW_{G , X , \mathfrak{s}}( x^{m_0} (x+v)^{m_1} \cdots (x+(p-1)v)^{m_{p-1}}) \\
\quad \quad &= e_{G}( H^+(X)/H^+(X)^G ) \sum_{j=0}^{p-1} c_j\left( -\frac{ (b_0+1) }{2} \, ; m_0-d_0 , \dots , m_{p-1} - d_{p-1} \right) \overline{SW}^{s_j}_{G}(X , \mathfrak{s})
\end{align*}
where $e_{G}( H^+(X)/H^+(X)^G) \in H_{G}^{b_+(X) - b_0}(pt)$ is the $G$-equivariant Euler class of $H^+(X)/H^+(X)^G$.
\end{theorem}
\begin{proof}
In the case where $X$ is closed, this is \cite[Theorem 6.12]{bar}. The proof in \cite{bar} works more generally for any abstract $\mathbb{Z}_p$-monopole map. Hence the result carries over without difficulty to the case where $X$ has psc boundary.
\end{proof}

Note that Theorem \ref{thm:eqre} is an equality in the equivariant cohomology ring, $H^*_{\mathbb{Z}_p}$, hence it can be regarded as an equality mod $p$. Theorem \ref{thm:eqre} does not imply an equality of integer-valued Seiberg--Witten invariants.

When $p=2$, Theorem \ref{thm:eqre} simplifies to
\[
SW_{G , X , \mathfrak{s}}( x^{m_0}(x+v)^{m_1}) = \left( \binom{m_1 - d_1}{\delta_0 - m_0} \overline{SW}^{s_0}_{G}(X , \mathfrak{s}) + \binom{m_0-d_0}{\delta_1 - m_1} \overline{SW}^{s_1}_G(X , \mathfrak{s}) \right) v^{m_0 + m_1 - \delta}
\]
where $\delta = d(X , \mathfrak{s})/2 = (d_0 + d_1) - b_+(X)-1$, $\delta_i = d(X , \mathfrak{s})^{s_i}/2 = d_i - b_+(X)^G - 1$.

\begin{proposition}\label{prop:ubcov}
Let $X$ be a compact, oriented, smooth $4$-manifold with psc boundary and $b_1(X) = 0$. Let $\pi : \widetilde{X} \to X$ be an unbranched covering of degree $d$. Let $\mathfrak{s}$ be a spin$^c$-structure on $X$ and let $\widetilde{\mathfrak{s}} = \pi^*(\mathfrak{s})$. Assume that $b_1(\widetilde{X}) = 0$. Let $Y = \partial X$ and $\widetilde{Y} = \partial \widetilde{X}$, oriented as ingoing boundaries. Then
\begin{align*}
\sigma(\widetilde{X}) &= d \sigma(X) + \rho(Y) \\
b_+(\widetilde{X}) &= d b_+(X) + (d-1) + \frac{1}{2}( b_0(\widetilde{Y}) - d b_0(Y) ) + \frac{1}{2}\rho(Y) \\
ind_{APS}(\widetilde{X} , \widetilde{\mathfrak{s}}) &= d \, ind_{APS}(X,\mathfrak{s}) + \delta(\widetilde{Y},\widetilde{\mathfrak{s}}) - d \delta(Y ,\mathfrak{s}) - \frac{1}{8}\rho(Y) \\
d(\widetilde{X} , \widetilde{\mathfrak{s}}) &= d \, d(X,\mathfrak{s}) - \frac{1}{2}(b_0(\widetilde{Y}) - d b_0(Y)) + 2(\delta(\widetilde{Y},\widetilde{\mathfrak{s}}) - d \delta(Y,\mathfrak{s})) - \frac{3}{4}\rho(Y)
\end{align*}
where $\rho(Y) = \eta_{sig}( \widetilde{Y} , \pi^*(g) ) - d \eta_{sig}(Y , g)$ denotes the rho invariant of the cover $\widetilde{Y} \to Y$.
\end{proposition}
\begin{proof}
Fix a metric $g$ on $X$ which is a product near the boundary and has positive scalar curvature on $Y$. The APS index theorem for the signature gives
\[
\sigma(X) = \frac{1}{3} \int_X p_1(X,g) + \eta_{sig}(Y,g)
\]
where $p_1(X,g)$ denotes the first Pontryagin form on $X$ with respect to $g$. Similarly
\[
\sigma(\widetilde{X}) = \frac{1}{3} \int_{\widetilde{X}} p_1(\widetilde{X}, \pi^*(g)) + \eta_{sig}(\widetilde{Y},\pi^*(g)).
\]
Since $p_1(\widetilde{X}) = \pi^*(p_1(X))$, we have $\int_{\widetilde{X}} p_1(\widetilde{X}) = d \int_X p_1(X)$ and hence $\sigma(\widetilde{X}) = d \sigma(X) + \rho(Y)$.

Let $r = b_0(Y)$, $\widetilde{r} = b_0(\widetilde{Y})$. Since $b_1(X) = 0$, $H^3(X ; \mathbb{R}) \cong H_1(X , Y ; \mathbb{R})$. The long exact sequence for $(X,Y)$ then implies that $b_3(X) = r-1$ if $r > 0$ or $b_3(X) = 0$ or $r=0$. Hence $\chi(X) = 2 + b_2(X) - r$. Similarly $b_1(\widetilde{X}) = 0$ implies that $\chi(\widetilde{X}) = 2 + b_2(\widetilde{X}) - \widetilde{r}$. But since $\widetilde{X} \to X$ is an unbranched covering, $\chi(\widetilde{X}) = d \chi(X)$ and this gives $b_2(\widetilde{X}) = d b_2(X) + 2(d-1) + \widetilde{r} - dr$. Then
\begin{align*}
b_+(\widetilde{X}) &= \frac{ b_2(\widetilde{X}) + \sigma(\widetilde{X}) }{2} \\
&= \frac{ d b_2(X) + 2(d-1) + \widetilde{r} - dr + d \sigma(X) + \rho(Y) }{2} \\
&= d b_+(X) + (d-1) + \frac{1}{2}(\widetilde{r} - dr) + \frac{1}{2}\rho(Y).
\end{align*}
Next, we have
\begin{align*}
\delta( \widetilde{X} , \widetilde{\mathfrak{s}}) &= \frac{ c(\widetilde{\mathfrak{s}})^2 - \sigma(\widetilde{X}) }{8} \\
&= \frac{ d c(\mathfrak{s})^2 - d \sigma(X) - \rho(Y) }{8} \\
&= d \, \delta(X , \mathfrak{s}) - \frac{1}{8}\rho(Y).
\end{align*}
Hence
\begin{align*}
ind_{APS}( \widetilde{X} , \widetilde{\mathfrak{s}}) &= \delta(\widetilde{X} , \widetilde{\mathfrak{s}}) + \delta(\widetilde{Y} , \widetilde{\mathfrak{s}}) \\
&= d \, \delta(X , \mathfrak{s}) - \frac{1}{8}\rho(Y) + \delta(\widetilde{Y} , \widetilde{\mathfrak{s}}) \\
&= d \, ind_{APS}(X , \mathfrak{s}) - \frac{1}{8}\rho(Y) + \delta(\widetilde{Y} , \widetilde{\mathfrak{s}}) - d \, \delta(Y , \mathfrak{s}).
\end{align*}
Lastly, we have
\begin{align*}
d(\widetilde{X} , \widetilde{\mathfrak{s}}) &= 2 ind_{APS}(\widetilde{X} , \widetilde{\mathfrak{s}}) - b_+(\widetilde{X}) - 1 \\
&= 2d \, ind_{APS}(X , \mathfrak{s}) - \frac{1}{4}\rho(Y) + 2(\delta(\widetilde{Y} , \widetilde{\mathfrak{s}}) - d \, \delta(Y , \mathfrak{s})) - d b_+(X) - d -\frac{1}{2}( \widetilde{r} - dr) - \frac{1}{2}\rho(Y) \\
&= 2d \, ind_{APS}(X , \mathfrak{s}) - \frac{3}{4}\rho(Y) + 2(\delta(\widetilde{Y} , \widetilde{\mathfrak{s}}) - d \, \delta(Y , \mathfrak{s})) - d b_+(X) - d -\frac{1}{2}( \widetilde{r} - dr) \\
&= d \, d(X , \mathfrak{s}) + 2(\delta(\widetilde{Y} , \widetilde{\mathfrak{s}}) - d \, \delta(Y , \mathfrak{s})) -\frac{1}{2}( \widetilde{r} - dr) - \frac{3}{4}\rho(Y).
\end{align*}

\end{proof}

The following result relates the reduced Seiberg--Witten invariants for a free action to the ordinary Seiberg--Witten invariants of the quotient space.
\begin{proposition}\label{prop:free}
Let $\widetilde{X}$ be a compact, oriented, smooth $4$-manifold with psc boundary and $b_1(\widetilde{X}) = 0$. Let $G = \mathbb{Z}_p$ act freely on $\widetilde{X}$ by orientation preserving diffeomorphisms which send each component of $\partial \widetilde{X}$ to itself and which preserves a psc metric on the boundary. Hence the quotient $X = \widetilde{X}/G$ also has psc boundary. Assume that $b_+(\widetilde{X})^G > 1$. Let $\widetilde{\mathfrak{s}}$ be a $G$-invariant spin$^c$-structure on $\widetilde{X}$ and let $s : G \to G_{\widetilde{\mathfrak{s}}}$ be a splitting. The splitting makes $\widetilde{\mathfrak{s}}$ into a $G$-equivariant spin$^c$-structure which descends to a spin$^c$-structure $\mathfrak{s}$ on $X$. Then we have an equality
\[
SW(X , \mathfrak{s}) = \overline{SW}_{\mathbb{Z}_p}^s( \widetilde{X} , \widetilde{\mathfrak{s}}).
\]
\end{proposition}
\begin{proof}
The case where $X$ has no boundary is given by \cite[Proposition 6.1]{bar} and the same argument easily carries over to the case of psc boundary.
\end{proof}

\subsection{An adjunction-type inequality}

We show that a simple application of the gluing formula \ref{thm:pscglue} yields an adjunction-type inequality. Suppose that $X$ is a compact, oriented, smooth $4$-manifold with $b_1(X) = 0$ and $b_+(X) > 1$. Suppose that $X = N \cup_Y X_0$ where $N$ is negative definite and $Y$ is a rational homology $3$-sphere which admits a metric of positive scalar curvature. 

\begin{proposition}\label{prop:switch}
Let $\mathfrak{s}$ be any spin$^c$-structure on $X$. Then $SW(X,\mathfrak{s}) = SW(X , x \otimes \mathfrak{s})$ for any $x \in H_2(N) \cong H^2(N,\partial N)$, provided $d(X,\mathfrak{s}), d(X , x \otimes \mathfrak{s}) \ge 0$.
\end{proposition}
\begin{proof}
For any spin$^c$-structure $\mathfrak{s}$ on $X$, Theorem \ref{thm:pscglue} yields $SW(X,\mathfrak{s}) = SW(X_0 , \mathfrak{s}|_{X_0})$, provided $d(X , \mathfrak{s} ) \ge 0$. Moreover, $d(X,\mathfrak{s}) \ge d(X_0 , \mathfrak{s}_0)$. Now the result follows by simply noting that $( x \otimes \mathfrak{s} )|_{X_0} \cong \mathfrak{s}|_{X_0}$.
\end{proof}

\begin{corollary}
Let $X$ be as above and suppose also that $X$ has simple type. If $SW(X , \mathfrak{s}) \neq 0$, then
\[
|\langle c(\mathfrak{s}) , x \rangle | + x^2 \le 0 \text{ for all } x \in H^2(N,\partial N). 
\]
\end{corollary}
\begin{proof}
Since $SW(X , \mathfrak{s}) \neq 0$ and $X$ has simple type, we must have $d(X , \mathfrak{s}) = 0$. Then it follows that
\begin{align*}
d(X , x \otimes \mathfrak{s}) &= d(X , x \otimes \mathfrak{s}) - d(X,\mathfrak{s}) \\
&= \langle c(\mathfrak{s}) , x \rangle + x^2.
\end{align*}
But since $X$ has simply type, Proposition \ref{prop:switch} implies that $\langle c(\mathfrak{s}) , x \rangle + x^2 \le 0$. Replacing $x$ by $-x$, we also get $-\langle c(\mathfrak{s}) , x \rangle + x^2 \le 0$, hence $|\langle c(\mathfrak{s}) , x \rangle | + x^2 \le 0$.
\end{proof}

\section{Branched covers}\label{sec:branch}

\subsection{Existence}\label{sec:ex}

We consider cyclic branched covers of degree $p$, where $p$ is a prime and the branch locus is a union of spheres of negative self-intersection. 

Let $X$ be a closed, oriented, smooth $4$-manifold with $b_1(X) = 0$. Let $S_1, \dots , S_r$ be a disjoint collection of embedded spheres in $X$, where each sphere $S_i$ has negative self-intersection $[S_i]^2 = -n_i$. We assume throughout that $r \ge 1$. Let $X_0$ be the complement of an open tubular neighbourhood $\nu S$ of $S = S_i \cup \cdots \cup S_r$. Then $X$ has psc boundary $Y = Y_1 \cup \cdots \cup Y_r$, $Y_i = L(n_i,1)$ (we orient the boundary of $X$ so that it is ingoing). Then $X$ is obtained from $X_0$ by attaching a copy of $X(-n_i)$ to $Y_i$, where we recall that $X(-n_i)$ denotes the unit disc bundle over $S^2$ with Euler class $-n_i$. 

Suppose that $\pi : \widetilde{X} \to X$ is a cyclic branched cover of degree $p$, where $p$ is prime and with branch locus $S$. Let $\widetilde{S} = \widetilde{S}_1 \cup \cdots \cup \widetilde{S}_r$ be the preimage of $S$ in $\widetilde{X}$. Then $\widetilde{S}_i$ is a sphere with self-intersection $-n_i/p$. Define $\widetilde{X}_0$ to be the complement of $\pi^{-1}(\nu S)$. Then $\pi : \widetilde{X}_0 \to X_0$ is an unbranched cover, $\widetilde{X}_0$ has psc boundary $\widetilde{Y} = \widetilde{Y}_1 \cup \cdots \cup \widetilde{Y}_r$, $\widetilde{Y}_i = \pi^{-1}(Y_i) = L(n_i/p , 1)$ and $\widetilde{X}$ is obtained from $\widetilde{X}_0$ by attaching a copy of $X(-n_i/p)$ to $\widetilde{Y}_i$. Let $G = \mathbb{Z}_p = \langle \sigma \rangle$ be the group over covering transformations of the branched cover. So $\sigma$ acts on $\widetilde{X}$ by orientation preserving diffeomorphism and with fixed point set $\widetilde{S}$. Then $\sigma$ must act on the normal bundle of $\widetilde{S}_i$ as a rotation by some angle $\theta_i \in 2\pi \mathbb{R}/\mathbb{Z}$, which can be written in the form $\theta_i = 2\pi \phi_i/p$, for some $\phi_i \in \mathbb{Z}_p^* = \mathbb{Z}_p \setminus \{0\}$. We call $\{ \phi_i \}$ the {\em normal weights} of the branched covering.

Suppose we are given a $4$-manifold $X$, a collection of disjoint embedded spheres $S_i$ with negative self-intersection and a collection of weights $\phi_i \in \mathbb{Z}_p^*$. We seek conditions under which a cyclic branched cover $\widetilde{X} \to X$ of degree $p$ exists with branch set $\{S_i\}$ and normal weights $\{ \phi_i \}$. Additionally, we seek conditions under which $b_1(\widetilde{X}) = 0$.

By excision and Poincar\'e--Lefschetz, we have
\[
H_j(X,X_0 ; \mathbb{Z}) \cong \bigoplus_i H_j( X(-n_i) , \partial X(-n_i) ; \mathbb{Z}) \cong \bigoplus_i H^{4-j}( X(-n_i) ; \mathbb{Z}) \cong \bigoplus_i H^{4-j}(S_i ; \mathbb{Z}).
\]
Using this, the long exact sequence for the pair $(X , X_0)$ gives
\[
\cdots \to H_2(X ; \mathbb{Z}) \buildrel \varphi \over \longrightarrow \bigoplus_i \mathbb{Z} \buildrel \partial \over \longrightarrow H_1(X_0 ; \mathbb{Z}) \to H_1(X ; \mathbb{Z})  \to 0.
\]
Let $x_1,\dots , x_r$ denote a basis for $\bigoplus_i \mathbb{Z}$. Then the map $\partial$ sends $x_i$ to $\mu_i$, the meridian of $S_i$. Furthermore, the map $\varphi$ is given by $\varphi(x) = \sum_i \langle x , [S_i] \rangle x_i$. A similar exact sequence exists with other coefficient groups in place of $\mathbb{Z}$.

\begin{lemma}\label{lem:bre}
In order for a degree $p$ cyclic branched cover $\widetilde{X} \to X$ to exist with branch set $\{S_i\}$ and normal weights $\{ \phi_i \}$ it is necessary and sufficient that $\sum_i \phi_i [S_i] = 0 \in H^2(X ; \mathbb{Z}_p)$.
\end{lemma}
\begin{proof}
The existence of $\widetilde{X} \to X$ is equivalent to the existence of a homomorphism $\psi : H_1(X_0 ; \mathbb{Z}_p) \to \mathbb{Z}_p$ such that $\psi( \mu_i ) = \phi_i$ for each $i$. Consider the long exact sequence associated to the pair $(X , X_0)$ with $\mathbb{Z}_p$-coefficients:
\[
\cdots \to H_2(X ; \mathbb{Z}_p) \buildrel \varphi \over \longrightarrow \bigoplus_i \mathbb{Z}_p \buildrel \partial \over \longrightarrow H_1(X_0 ; \mathbb{Z}_p) \to H_1(X ; \mathbb{Z}_p)  \to 0.
\]
For each $x \in H_2(X ; \mathbb{Z}_p)$, we have $\varphi(x) = \sum_i \langle x , [S]_i \rangle x_i$ and hence $0 = \partial \varphi(x) = \sum_i \langle x , [S_i] \rangle \mu_i$. Applying $\psi$ gives $0 = \langle x , [S_i] \rangle \psi(\mu_i) = \langle x , \sum_i \phi_i [S_i] \rangle$. Since this holds for all $x \in H_2(X ; \mathbb{Z}_p)$, we must have $\sum_i \phi_i [S_i] = 0 \in H_2(X ; \mathbb{Z}_p)$. Conversely if $\sum_i \phi_i [S_i] = 0 \in H^2(X ; \mathbb{Z}_p)$, then choosing integer lifts $\widehat{\phi}_i$ of $\phi_i$, we have that $\sum_i \widehat{\phi}_i [S_i] \in H^2(X ; \mathbb{Z})$ is divisible by $p$. Choose $\alpha \in H^2(X ; \mathbb{Z})$ with $p\alpha = \sum_i \widehat{\phi}_i [S_i]$. Then $p \alpha |_{X_0} = 0$. Hence $\alpha|_{X_0}$ is $p$-torsion. The class $\alpha|_{X_0}$ corresponds to a line bundle $A \to X_0$ with $A^p$ trivial. Choose a non-vanishing section $\tau$ of $A^p$. Then we can define a branched cover by taking $\widetilde{X}_0 = \{ a \in A \; | \; a^p = \tau \}$ and attaching a copy of $X(-n_i/p)$ to each boundary component.
\end{proof}

\begin{lemma}\label{lem:b10}
Suppose that the branched cover $\widetilde{X} \to X$ exists. If $H_1(X_0 ; \mathbb{Z})$ is finite and $H_1(X_0 ; \mathbb{Z}_p) \cong \mathbb{Z}_p$, then $b_1(\widetilde{X}) = 0$.
\end{lemma}
\begin{proof}
Our proof is based on \cite[\textsection 3]{rok}. First note that since the boundary of $\widetilde{X}_0$ is a union of rational homology spheres, we have that $b_1(\widetilde{X}) = b_1(\widetilde{X}_0) = 0$. So it suffices to show that $b_1(\widetilde{X}_0) = 0$. Since $H_1(X_0 ; \mathbb{Z}_p) \cong \mathbb{Z}_p$, we have that $H_1(X_0 ; \mathbb{Z}) \cong A \oplus \mathbb{Z}_{p^m}$, where $A$ is finite of order coprime to $p$ and $m \ge 1$. The projection to $\mathbb{Z}_{p^m}$ defines a homomorphism $H_1(X_0 ; \mathbb{Z}) \to \mathbb{Z}_{p^m}$ and hence a degree $p^m$ cyclic cover $X'_0 \to X_0$. Since $\widetilde{X}_0$ is a quotient of $X'_0$, it suffices to show that $b_1(X'_0) = 0$.

Let $\pi = \pi_1(X_0)$, $\pi' = \pi_1(X'_0)$. Then we have a short exact sequence
\[
1 \to \pi' \to \pi \to \mathbb{Z}_{p^m} \to 0.
\]
From the Lyndon--Hochschild--Serre spectral sequence we get an exact sequence
\[
0 \to H_0( \mathbb{Z}_{p^m} ; H_1(X'_0 ; \mathbb{Z})) \to H_1(X_0 ; \mathbb{Z}) \to \mathbb{Z}_{p^m} \to 0.
\]
Since the kernel of $H_1(X_0 ; \mathbb{Z}) \to \mathbb{Z}_{p^m}$ is $A$, we have that $H_0( \mathbb{Z}_{p^m} ; H_1(X'_0 ; \mathbb{Z})) \cong A$ is finite with order coprime to $p$. Let $M = H_1(X'_0 ; \mathbb{Z})$. Then $H_0( \mathbb{Z}_{p^m} ; H_1(X'_0 ; \mathbb{Z})) = H_0( \mathbb{Z}_{p^m} ; M ) \cong M/(1-\tau)$, where $\tau : M \to M$ generates the action of $\mathbb{Z}_{p^m}$ on $M$. Since $M/(1-\tau)$ is finite of order coprime to $p$, \cite[Lemma 3.3]{rok} implies that $M$ is finite. Hence $b_1(X'_0) = 0$.
\end{proof}

\begin{lemma}\label{lem:1p}
Suppose that $H_1(X ; \mathbb{Z})$ is finite of order coprime to $p$. Suppose that any $r-1$ of $[S_1], \dots , [S_r]$ are linearly indepenent in $H^2(X ; \mathbb{Z}_p)$ but that $[S_1], \dots , [S_r]$ are linearly dependent in $H^2(X ; \mathbb{Z}_p)$. Then $H_1(X_0 ; \mathbb{Z})$ is finite and $H_1(X_0 ; \mathbb{Z}_p) \cong \mathbb{Z}_p$.
\end{lemma}
\begin{proof}
Recall the exact sequence
\[
\cdots \to H_2(X ; \mathbb{Z}_p) \buildrel \varphi \over \longrightarrow \mathbb{Z}^r_p \buildrel \partial \over \longrightarrow H_1(X_0 ; \mathbb{Z}_p) \to H_1(X ; \mathbb{Z}_p)  \to 0
\]
The assumptions on $[S_1], \dots , [S_r]$ implies that the image of $\varphi$ in $\mathbb{Z}_p^r$ has dimension $r-1$, hence we have a short exact sequence
\[
0 \to \mathbb{Z}_p \to H_1(X_0 ; \mathbb{Z}_p) \to H_1(X ; \mathbb{Z}_p)  \to 0.
\]
Our assumption on $H_1(X ; \mathbb{Z})$ implies that $H_1(X ; \mathbb{Z}_p) = 0$, hence $H_1(X_0 ; \mathbb{Z}_p) = \mathbb{Z}_p$. 
\end{proof}

Putting together Lemmas \ref{lem:bre}, \ref{lem:b10} and \ref{lem:1p}, we have shown:

\begin{proposition}
Let $X$ be a compact, oriented, smooth $4$-manifold. Let $S_1, \dots , S_r$ be disjoint embedded spheres of negative self-intersection. Let $p$ be a prime and suppose the following conditions hold:
\begin{itemize}
\item[(1)]{$H_1(X ; \mathbb{Z})$ is finite with order coprime to $p$.}
\item[(2)]{Any $r-1$ of $[S_1] , \dots , [S_r]$ are linearly independent in $H^2(X ; \mathbb{Z}_p)$.}
\item[(3)]{$\sum_i \phi_i [S_i] = 0 \in H^2(X ; \mathbb{Z}_p)$, for some $\phi_i \in \mathbb{Z}_p^*$.}
\end{itemize}
Then there exists a cyclic degree $p$ branched cover $\widetilde{X} \to X$ with branch set $\{ S_i \}$ and normal weights $\{ \phi_i \}$. Furthermore, $b_1(\widetilde{X} ) = 0$.
\end{proposition}

\subsection{Spin$^c$-structures on branched covers}\label{sec:spinc}

Let $\pi : \widetilde{X} \to X$ be a branched cover as in Section \ref{sec:ex} and define $\widetilde{X}_0, X_0$ as before. We study the relationship between spin$^c$-structures on $X$ and $\mathbb{Z}_p$-invariant spin$^c$-structures on $\widetilde{X}$.

\begin{lemma}\label{lem:equiv}
Let $\mathfrak{s}, \mathfrak{s}'$ be spin$^c$-structures on $X$. Then $\mathfrak{s}|_{X_0} \cong \mathfrak{s}'|_{X_0}$ if and only if $\mathfrak{s}' = \sum_i m_i [S]_i + \mathfrak{s}$ for some integers $m_i$.
\end{lemma}
\begin{proof}
We have that $\mathfrak{s}' = L + \mathfrak{s}$ for some $L \in H^2(X ; \mathbb{Z})$. Then $\mathfrak{s}|_{X_0} \cong \mathfrak{s}'|_{X_0}$ if and only if $L|_{X_0} = 0$. Hence $L$ lies in the image of $H^2(X , X_0 ; \mathbb{Z}) \to H^2(X ; \mathbb{Z})$. By excision $H^2(X , X_0 ; \mathbb{Z}) \cong \bigoplus_i H^2( X(-n_i) , \partial X(-n_i) ; \mathbb{Z})$ and the image of $H^2(X(-n_i) , \partial X(-n_i) ; \mathbb{Z})$ in $H^2(X ; \mathbb{Z})$ is generated by $[S_i]$.
\end{proof}

Let $Spin^c(X)$ denote the set of isomorphisms classes of spin$^c$-structures on $X$ and similarly define $Spin^c(X_0)$. Let $Spin^c_{sharp}(X) \subseteq Spin^c(X)$ be the set of sharp spin$^c$-structures on $X$, that is, $| \langle c(\mathfrak{s}) , [S_i] \rangle | + [S_i]^2 \le 0$ for all $i$. Define an equivalence relation $\sim$ on $Spin^c(X)$ by setting $\mathfrak{s}_1 \sim \mathfrak{s}_2$ if $\mathfrak{s}_1 |_{X_0} \cong \mathfrak{s}_2 |_{X_0}$. By Lemma \ref{lem:equiv} this is equivalent to $\mathfrak{s}_2 = \sum_i m_i [S]_i + \mathfrak{s}_1$ for some integers $m_i$.

\begin{proposition}\label{prop:spincext}
The restriction map $Spin^c(X) \to Spin^c(X_0)$, $\mathfrak{s} \mapsto \mathfrak{s}|_{X_0}$ induces a bijection between $Spin^c(X_0)$ and $Spin^c(X)/\! \! \sim$. Moreover, each $\mathfrak{s}_0 \in Spin^c(X_0)$ extends to a sharp spin$^c$-structure on $X$. If $\mathfrak{s}$ is a sharp extension of $\mathfrak{s}_0$, then any other sharp extension of $\mathfrak{s}_0$ has the form $\mathfrak{s}' = \sum_i \epsilon_i[S_i] + \mathfrak{s}$, where $\epsilon_i \in \{0,1\}$ if $\langle c(\mathfrak{s}) , [S_i] \rangle = n_i$, $\epsilon_i \in \{0, -1\}$ if $\langle c(\mathfrak{s}) , [S_i] \rangle = -n_i$ and $\epsilon_i = 0$ otherwise. If $\mathfrak{s}, \mathfrak{s}'$ are sharp extensions of $\mathfrak{s}_0$, then $SW(X , \mathfrak{s}) = SW(X , \mathfrak{s}')$.
\end{proposition}
\begin{proof}
Follows easily from Lemma \ref{lem:equiv}.
\end{proof}

Let $Spin^c( \widetilde{X} )^G \subseteq Spin^c(\widetilde{X})$ denote the isomorphism classes of $G$-invariant spin$^c$-structures on $\widetilde{X}$ and similarly define $Spin^c(\widetilde{X}_0)^G$. Then since $G$ preserves the classes $[\widetilde{S}_i] \in H^2( \widetilde{X} ; \mathbb{Z})$, Lemma \ref{lem:equiv} applied to $\widetilde{X}$ implies that if $\widetilde{\mathfrak{s}}_1, \widetilde{\mathfrak{s}}_2 \in Spin^c(\widetilde{X})^G$, then $\widetilde{\mathfrak{s}}_1 |_{\widetilde{X}_0} \cong \widetilde{\mathfrak{s}}_2 |_{\widetilde{X}_0}$ if and only if $\widetilde{\mathfrak{s}}_2 = \sum_i m_i [\widetilde{S}_i] + \widetilde{\mathfrak{s}}_1$ for some integers $m_i$. Likewise, Proposition \ref{prop:spincext} applied to $\widetilde{X}$ shows that every $\widetilde{\mathfrak{s}}_0 \in Spin^c( \widetilde{X}_0)^G$ has an extension to a $G$-invariant sharp spin$^c$-structure $\widetilde{\mathfrak{s}}$ on $\widetilde{X}$ and any other sharp extension has the form $\sum_i \epsilon_i[\widetilde{S}_i] + \widetilde{\mathfrak{s}}$, where $\epsilon_i \in \{ 0, 1\}$ if $\langle c(\widetilde{\mathfrak{s}}) , [\widetilde{S}_i] \rangle = n_i/p$, $\epsilon_i \in \{ 0, -1\}$ if $\langle c(\widetilde{\mathfrak{s}}) , [\widetilde{S}_i] \rangle = -n_i/p$ and $\epsilon_i = 0$ otherwise.

Let $Spin^c_G(\widetilde{X})$ denote the isomorphism classes of $G$-equivariant spin$^c$-structures on $\widetilde{X}$ and similarly define $Spin^c_G(\widetilde{X}_0)$. An element of $Spin^c_G(\widetilde{X})$ can be thought of as a pair $(\widetilde{\mathfrak{s}} , s)$, where $\widetilde{\mathfrak{s}} \in Spin^c(\widetilde{X})^G$ and $s : G \to G_{\widetilde{\mathfrak{s}}}$ is a splitting of $G_{\widetilde{\mathfrak{s}}} \to G$. In particular, the forgetful map $Spin^c_G(\widetilde{X}) \to Spin^c(\widetilde{X})^G$ is surjective and there are $p$ elements of $Spin^c_G(\widetilde{X})$ which map to any given $\widetilde{\mathfrak{s}} \in Spin^c(\widetilde{X})^G$. A similar statement holds for the forgetful map $Spin^c_G(\widetilde{X}_0) \to Spin^c(\widetilde{X}_0)^G$.

Since the $G$-action is free on $\widetilde{X}_0$ with quotient $X_0$, any $G$-equivariant spin$^c$-structure on $\widetilde{X}_0$ descends to $X_0$. Conversely, any spin$^c$-structure on $X_0$ pulls back to a $G$-equivariant spin$^c$-structure on $\widetilde{X}_0$. Hence we have a natural bijection $Spin^c_G(\widetilde{X}_0) \cong Spin^c(X_0)$. Since $\widetilde{X}_0 \to X_0$ is a non-trivial unbranched cyclic cover of degree $p$, there is a corresponding flat line bundle $A \to X_0$ of order $p$.

\begin{lemma}
The pullback map $\pi^* : Spin^c(X_0) \to Spin^c(\widetilde{X}_0)^G$ is surjective. Moreover $\pi^*(\mathfrak{s}_1) \cong \pi^*(\mathfrak{s}_2)$, if and only if $\mathfrak{s}_2 = A^{-j} \otimes \mathfrak{s}_1$ for $j \in \mathbb{Z}_p$. Furthermore, for any $\mathfrak{s}_0 \in Spin^c(X_0)$, the spin$^c$-structures $\{ A^{-j} \otimes \mathfrak{s} \}_{j=0}^{p-1}$ are distinct.
\end{lemma}
\begin{proof}
Under the natural isomorphism $Spin^c_G(\widetilde{X}_0) \cong Spin^c(X_0)$, the pullback map $\pi^*$ becomes identified with the forgetful map $Spin^c_G(\widetilde{X}_0) \to Spin^c(\widetilde{X}_0)^G$. This map is surjective and any two equivariant spin$^c$-structures on $\widetilde{X}_0$ map to the same underlying spin$^c$-structure if and only if the corresponding spin$^c$-structures on $X_0$ differ by a power of $A$. The distinctness of the spin$^c$-structures $\{ A^{-j} \otimes \mathfrak{s} \}_{j=0}^{p-1}$ follows from the fact that $Spin^c(X_0)$ is a torsor over $H^2(X_0 ; \mathbb{Z})$ and the fact that $[A]$ has order $p$ ($[A]$ is non-trivial since $\widetilde{X}_0 \to X_0$ is a non-trivial covering).
\end{proof}

\subsection{Seiberg--Witten invariants}\label{sec:sw}

The relationship between spin$^c$-structures on $X$ and $G$-invariant spin$^c$-structures on $\widetilde{X}$ is as follows. Since the Seiberg--Witten invariants of non-sharp spin$^c$-structures are determined by the Seiberg--Witten invariants of the sharp spin$^c$-structures, it suffices to restrict attention to these. Let $\mathfrak{s}$ be a sharp spin$^c$-structure on $X$. Let $\mathfrak{s}_0 = \mathfrak{s}|_{X_0}$ be its restriction to $X_0$. Let $\widetilde{\mathfrak{s}}_0 = \pi^*(\mathfrak{s}_0)$ be its lift to $\widetilde{X}_0$ and let $\widetilde{\mathfrak{s}}$ be a sharp extension of $\widetilde{\mathfrak{s}}_0$ to $\widetilde{X}_0$. Since all sharp extensions have the same Seiberg--Witten invariants, it will not matter which choice of sharp extension we take.

The gluing theorem (Theorem \ref{thm:pscglue}) says that $SW(\widetilde{X} , \widetilde{\mathfrak{s}}) = SW(\widetilde{X}_0 , \widetilde{\mathfrak{s}}_0)$ and $SW(X , \mathfrak{s}) = SW(X_0 , \mathfrak{s}_0)$, since $\widetilde{\mathfrak{s}}$ and $\mathfrak{s}$ are sharp. Hence to obtain a relation between $SW(\widetilde{X} , \widetilde{\mathfrak{s}})$ and $SW(X , \mathfrak{s})$, it will suffice to relate $SW(\widetilde{X}_0 , \widetilde{\mathfrak{s}}_0)$ and $SW(X_0 , \mathfrak{s}_0)$.

For $j \in \mathbb{Z}_p$, set $\mathfrak{s}_j = A^{-j} \otimes \mathfrak{s}_0 \in Spin^c(X_0)$. Then $\{ \mathfrak{s}_j \}$ are precisely the spin$^c$-structures which pull back to $\widetilde{\mathfrak{s}}_0$. Let $\{ \phi_i \}$ be the normal weights of $\sigma$. Choose integer lifts $\widehat{\phi}_i$ of each $\phi_i$. We assume that $H_1(X ; \mathbb{Z})$ is finite and has order coprime to $p$. Then there is a unique class $\alpha \in H^2(X ; \mathbb{Z})$ such that $p \alpha = \sum_i \widehat{\phi_i} [S_i] \in H^2(X ; \mathbb{Z})$. Then $\alpha|_{X_0} = [A]$ and hence $(-j \alpha + \mathfrak{s})|_{X_0} = \mathfrak{s}_j$. In general, $-j\alpha + \mathfrak{s}$ will not be sharp, but we can choose integers $\{ m_i^j\}$ such that $\alpha_j + \mathfrak{s}$ is sharp, where $\alpha_j = -j \alpha + \sum_i m_i^j [S_i]$. The gluing theorem gives 
\begin{equation}\label{equ:x0x}
SW(X_0 , \mathfrak{s}_{j}) = SW(X , \alpha_j + \mathfrak{s}).
\end{equation}
Furthermore, Proposition \ref{prop:free} gives
\begin{equation}\label{equ:swbarx0}
\overline{SW}_{\mathbb{Z}_p}^{s_j}( \widetilde{X}_0 , \widetilde{\mathfrak{s}}_0) = SW(X_0 , \mathfrak{s}_j),
\end{equation}
where $s_j$ is the splitting of $G_{\widetilde{\mathfrak{s}}_0} \to G$ corresponding to $\mathfrak{s}_j$. More precisely, since $\widetilde{\mathfrak{s}}_0 = \pi^*(\mathfrak{s}_0)$ is defined as a pullback of $\mathfrak{s}_0$, it naturally has the structure of a $G$-equivariant spin$^c$-structure. Let $s_0$ be the splitting corresponding to this equivariant structure. Then $s_j : G \to G_{\widetilde{\mathfrak{s}}_0}$ is given by $s_j(\sigma) = \omega^{-j} s_0(\sigma)$, $\omega = e^{2\pi i/p}$.

Using this and Theorem \ref{thm:eqre}, the equivariant Seiberg--Witten invariants of $(\widetilde{X}_0 , \widetilde{\mathfrak{s}}_0)$ can be expressed in terms of Seiberg--Witten invariants of $X_0$ (or of $X$ using Equation \ref{equ:x0x}).

\begin{theorem}\label{thm:sweqcov}
For any $m_0, \dots , m_{p-1} \ge 0$ we have
\begin{align*}
& SW_{G , \widetilde{X}_0 , \widetilde{\mathfrak{s}}_0}( x^{m_0} (x+v)^{m_1} \cdots (x+(p-1)v)^{m_{p-1}}) \\
\quad \quad &= e_{G}( H^+( \widetilde{X}_0)/H^+(X_0) ) \sum_{j=0}^{p-1} c_j\left( -\frac{ (b_0+1) }{2} \, ; m_0-d_0 , \dots , m_{p-1} - d_{p-1} \right) SW(X_0 , \mathfrak{s}_j)
\end{align*}
where $b_0 = b_+(X)$ and for each $j \in \mathbb{Z}_p$, $d_j = ind_{APS}( X_0 , \mathfrak{s}_j)$.
\end{theorem}
\begin{proof}
This follows from Theorem \ref{thm:eqre} and Equation (\ref{equ:swbarx0}). Note also that $d_j = ind_{APS}( \widetilde{X}_0 , \widetilde{\mathfrak{s}}_0)^{s_j G} = ind_{APS}(X_0 , \mathfrak{s}_0)$ because the $G$-action is free on $\widetilde{X}_0$.
\end{proof}

Since the equivariant Seiberg--Witten invariants map to the ordinary Seiberg--Witten invariants (mod $p$) under the forgetful map $H^*_G(pt) \to H^*(pt) \cong \mathbb{Z}_p$, Theorem \ref{thm:sweqcov} implies a formula for $SW(\widetilde{X}_0 , \widetilde{\mathfrak{s}}_0)$. Define $e \in \mathbb{Z}_p^*$ by 
\[
e_G( H^+(\widetilde{X}_0)/H^+(X_0) ) = e v^{(b_+(\widetilde{X}_0) - b_+(X_0) )/2}.
\]
Further, define $\mu_j( n ; n_0 , \dots , n_{p-1}) \in \mathbb{Z}_p$ by
\[
c_j( n ; n_0 , \dots , n_{p-1}) = \mu_j(n ; n_0 , \dots , n_{p-1}) v^{\sum_i n_i - n}.
\]
Comparing with the definition of $c_j$, it follows that
\[
\mu_j(n ; n_0 , \dots , n_{p-1}) = \sum_{k_i} \prod_{i | i \neq j} \binom{n_i}{k_i} (i-j)^{n_i-k_i}
\]
where the sum is over non-negative integers $k_0 , \dots , \hat{k}_j , \dots , k_{p-1}$ such that $k_0 + \cdots + k_{p-1} = n-n_j$.

\begin{theorem}\label{thm:swcov}
Suppose that $d(\widetilde{X}_0 , \widetilde{\mathfrak{s}}_0) = 2m \ge 0$ is even and non-negative. Then
\[
SW(\widetilde{X}_0 , \widetilde{\mathfrak{s}}_0) = e \sum_{j=0}^{p-1} \mu_j\left( -\frac{ (b_0+1) }{2} \, ; m_0-d_0 , \dots , m_{p-1} - d_{p-1} \right) SW(X_0 , \mathfrak{s}_j) \; ({\rm mod} \; p),
\]
where $b_0 = b_+(X)$, $d_j = ind_{APS}( X_0 , \mathfrak{s}_j)$ and $m_0, \dots , m_{p-1}$ are any non-negative integers such that $m_0 + \cdots + m_{p-1} = m$.
\end{theorem}
\begin{proof}
If $m_0 + \cdots + m_{p-1} = m$, then 
\[
SW( \widetilde{X}_0 , \widetilde{\mathfrak{s}}_0) = SW_{G , \widetilde{X}_0 , \widetilde{\mathfrak{s}}_0}( x^{m_0}(x+v)^{m_1} \cdots (x+(p-1)v)^{m_{p-1}}) \in H^0_{\mathbb{Z}_p}(pt) = \mathbb{Z}_p.
\]
The result now follows from Theorem \ref{thm:sweqcov}.
\end{proof}

If $\widetilde{\mathfrak{s}}$ is a sharp extension of $\widetilde{\mathfrak{s}}_0$, then Theorem \ref{thm:swcov} gives
\[
SW(\widetilde{X} , \widetilde{\mathfrak{s}}) = e \sum_{j=0}^{p-1} \mu_j\left( -\frac{ (b_0+1) }{2} \, ; m_0-d_0 , \dots , m_{p-1} - d_{p-1} \right) SW(X , \alpha_j + \mathfrak{s}) \; ({\rm mod} \; p).
\]
We have thus achieved our goal of expressing the Seiberg--Witten invariants of $\widetilde{X}$ (for $G$-invariant spin$^c$-structures) in terms of the Seiberg--Witten invariants of $X$.

An important feature of Theorem \ref{thm:swcov} is that the formula must hold for all choices of non-negative integers $m_0, \dots , m_{p-1}$ such that $m_0 + \cdots + m_{p-1} = d(\widetilde{X}_0 , \widetilde{\mathfrak{s}}_0)/2$. We will see that this imposes strong restrictions on how large $d(\widetilde{X}_0 , \widetilde{\mathfrak{s}}_0)$ can be. In fact, a similar restriction holds whether or not $d(\widetilde{X}_0 , \widetilde{\mathfrak{s}}_0)$ is even.

\begin{proposition}\label{prop:indep}
Suppose that $d(\widetilde{X}_0 , \widetilde{\mathfrak{s}}_0) \ge 0$ and set $m = \lceil d(\widetilde{X}_0 , \widetilde{\mathfrak{s}}_0)/2 \rceil$. Let $m_0, \dots , m_{p-1}$ be non-negative integers such that $m_0 + \cdots + m_{p-1} = m$. The value of
\[
\sum_{j=0}^{p-1} \mu_j\left( -\frac{ (b_0+1) }{2} \, ; m_0-d_0 , \dots , m_{p-1} - d_{p-1} \right) SW(X_0 , \mathfrak{s}_j) \; ({\rm mod} \; p)
\]
does not depend on the choice of $m_0, \dots , m_{p-1}$.
\end{proposition}
\begin{proof}
In the case that $d(\widetilde{X}_0 , \widetilde{\mathfrak{s}}_0) = 2m$ is even, this follows from Theorem \ref{thm:swcov} and the fact that $e \neq 0 \; ({\rm mod} \; p)$. Now suppose that $d(\widetilde{X}_0 , \widetilde{\mathfrak{s}}_0) = 2m-1 > 0$ is odd and positive. Then $SW_{G , \widetilde{X}_0 , \widetilde{\mathfrak{s}}_0}( x^k ) \in H^{2k-(2m-1)}_{\mathbb{Z}_p}(pt) = 0$ for $k < m$. Now if $m_0 + \cdots + m_{p-1} = m$, then it follows that the value of
\[
SW_{G , \widetilde{X}_0 , \widetilde{\mathfrak{s}}_0}( x^{m_0} (x+v)^{m_1} \cdots (x+(p-1)v)^{m_{p-1}})  \in H^1_{\mathbb{Z}_p}(pt)
\]
does not depend on the choice of $m_0, \dots , m_{p-1}$. To see this, expand $x^{m_0} (x+v)^{m_1} \cdots (x+(p-1)v)^{m_{p-1}}$ in terms of powers of $x$. Only the leading term $x^m$ contributes. Now the result follows from Theorem \ref{thm:sweqcov}.
\end{proof}

\begin{proposition}\label{prop:dimbound}
Suppose that $X$ has simple type. For any spin$^c$-structure $\mathfrak{s}$ on $X$, let $k$ be the number of $j \in \mathbb{Z}_p$ such that $SW(X , \alpha_j + \mathfrak{s}) \neq 0 \; ({\rm mod} \; p)$. Then
\begin{itemize}
\item[(1)]{If $k > 0$, then $d(\widetilde{X}_0 , \widetilde{\mathfrak{s}}_0) \le 2k-2$.}
\item[(2)]{If $d(\widetilde{X}_0 , \widetilde{\mathfrak{s}}_0) = 2k-2$ and $k > 1$, then $\widetilde{X}_0$ does not have simple type.}
\end{itemize}

\end{proposition}
\begin{proof}
First note the following two properties of $\mu_j( n ; n_0 , \dots , n_{p-1})$:
\begin{itemize}
\item[(a)]{If $n_j = n$, then the only solution to $k_0 + \cdots + \widehat{k}_j + \cdots + k_{p-1} = n-n_j = 0$ is $k_0 = \cdots = 0$. Hence
\[
\mu_j( n ; n_0 , \dots , n_{p-1}) = \prod_{i | i \neq j } (i-j)^{n_i} \neq 0 \; ({\rm mod} \; p).
\]
}
\item[(b)]{If $n_j > n$, then there are no solutions to $k_0 + \cdots + \widehat{k}_j + \cdots + k_{p-1} = n-n_j$ with $k_i \ge 0$, hence $\mu_j( n \; n_0 , \dots , n_{p-1}) = 0$.}
\end{itemize}

Let $m = \lceil d(\widetilde{X}_0 , \widetilde{\mathfrak{s}}_0)/2 \rceil$ and set
\[
S(m_0, \dots , m_{p-1}) = \sum_{j=0}^{p-1} \mu_j\left( -\frac{ (b_0+1) }{2} \, ; m_0-d_0 , \dots , m_{p-1} - d_{p-1} \right) SW(X_0 , \mathfrak{s}_j) \; ({\rm mod} \; p)
\]
where $m_i \ge 0$, $m_0 + \cdots + m_{p-1} = m$. Then by Proposition \ref{prop:indep}, $S$ does not depend on the choice of $m_0, \dots , m_{p-1}$. Assume $k > 0$, otherwise there is nothing to show. Let $K = \{ j \in \mathbb{Z}_p \; | \; SW(X , \alpha_j + \mathfrak{s}) \neq 0 \; ({\rm mod} \; p) \}$. Then
\[
S(m_0, \dots , m_{p-1}) = \sum_{j \in K} \mu_j\left( -\frac{ (b_0+1) }{2} \, ; m_0-d_0 , \dots , m_{p-1} - d_{p-1} \right) SW(X_0 , \mathfrak{s}_j) \; ({\rm mod} \; p).
\]
If $j \in K$ then the simple type assumption implies that $\alpha_j + \mathfrak{s}$ is sharp and that $d(X , \alpha_j + \mathfrak{s}) = d(X_0 , \mathfrak{s}_j) = 0$. But $d(X_0 , \mathfrak{s}_j) = 2d_j - b_+(X) - 1$, hence $d_j = (b_+(X)+1)/2 = (b_0 + 1)/2$. 

Suppose now that $d(\widetilde{X}_0 , \widetilde{\mathfrak{s}}_0) \ge 2k-1$. Then $m \ge k$. So we can choose $m_0, \dots , m_{p-1}$ with $m_0 + \cdots + m_{p-1} = m$ and $m_j > 0$ for all $j \in K$. But for $j \in K$, $m_j -d_j > -(b_0+1)/2$ and hence $\mu_j( -(b_0+1)/2 ; m_0-d_0 , \dots , m_{p-1} - d_{p-1}) = 0$, by property (b). Hence $S(m_0 , \dots , m_{p-1} ) = 0$. On the other hand, for any given $j \in K$, we can choose $\{ m_i \}$ such that $m_0 + \cdots + m_{p-1} = m$, $m_i > 0$ for all $i \in K \setminus \{j\}$ and $m_j = 0$. Then properties (a) and (b) give:
\[
S(m_0 , \dots , m_{p-1}) = \mu_j \left( -\frac{b_0+1}{2} \, ; m_0-d_0 , \dots , m_j - d_j \right) SW(X_0 , \mathfrak{s}_j) \neq 0 \; ({\rm mod}\; p ),
\]
which is a contradiction. So $d(\widetilde{X}_0 , \widetilde{\mathfrak{s}}_0) \le 2k-2$.

Now suppose $d(\widetilde{X}_0 , \widetilde{\mathfrak{s}}_0) = 2m = 2k-2$ and $k > 1$. Then $d(\widetilde{X}_0 , \widetilde{\mathfrak{s}}_0) > 0$ and $SW(\widetilde{X}_0 , \widetilde{\mathfrak{s}}_0) = e S(m_0 , \dots , m_{p-1}) \; ({\rm mod} \; p)$ by Theorem \ref{thm:swcov}. Fix some $j \in K$. Since $m = k-1$, we can choose $m_i$ to be $m_i = 1$ if $i \in K \setminus \{j\}$ and $m_i = 0$ otherwise. The same argument as above gives $S(m_0 , \dots , m_{p-1}) \neq 0 \; ({\rm mod} \; p)$ and hence also $SW(\widetilde{X}_0 , \widetilde{\mathfrak{s}}_0) \neq 0 \; ({\rm mod} \; p)$. Thus $\widetilde{X}_0$ does not have simple type.
\end{proof}

We seek a formula for $d(\widetilde{X} , \widetilde{\mathfrak{s}})$ in terms of the branching data. Assume that $\mathfrak{s}, \widetilde{\mathfrak{s}}$ are sharp extensions of $\mathfrak{s}_0$ and $\widetilde{\mathfrak{s}}_0$. The quantities $c_i = \langle c(\mathfrak{s}) , [S_i] \rangle$, $\widetilde{c}_i = \langle c(\widetilde{\mathfrak{s}}_0) , [\widetilde{S}_i] \rangle$ will be of importance. Note that sharpness of $\mathfrak{s}, \widetilde{\mathfrak{s}}$ is equivalent to $| c_i | \le n_i$, $| \widetilde{c}_i | \le n_i/p$ for all $i$. It follow from Lemma \ref{lem:equiv} that $c_i, \widetilde{c}_i$ are uniquely determined by $\mathfrak{s}_0, \widetilde{\mathfrak{s}}_0$ except that in the case $|c_i | = n_i$, $|\widetilde{c}_i | = n_i/p$, the signs of $c_i, \widetilde{c}_i$ are undetermined. We can remove the ambiguity by imposing the conditions $-n_i < c_i \le n_i$, $-n_i/p < \widetilde{c}_i \le n_i/p$. Then $c_i, \widetilde{c}_i$ are uniquely determined by $\mathfrak{s}_0, \widetilde{\mathfrak{s}}_0$. 

\begin{lemma}\label{lem:deltaY}
We have
\[
\delta( Y_i , \mathfrak{s}_0|_{Y_i}) = -\frac{c_i^2}{8n_i} + \frac{1}{8}, \quad \delta( \widetilde{Y}_i , \widetilde{\mathfrak{s}}_0 |_{\widetilde{Y}_i}) = -\frac{ \widetilde{c}_i^2}{8(n_i/p)} + \frac{1}{8}.
\]
Moreover $\widetilde{c}_i$ is uniquely determined from $c_i,n_i,p$ by the requirements that 
\begin{itemize}
\item[(1)]{$-n_i/p < \widetilde{c}_i \le n_i/p$, and}
\item[(2)]{$ \dfrac{(c_i+n_i)}{2(n_i/p)} - \dfrac{(\widetilde{c}_i + (n_i/p))}{2(n_i/p)} \in \mathbb{Z}.$ }
\end{itemize}
\end{lemma}

\begin{proof}
By assumption, $\mathfrak{s}$ is sharp and hence
\[
\delta(Y_i , \mathfrak{s}_0 |_{Y_i}) = \delta(Y_i , \mathfrak{s}|_{Y_i}) = \delta( X(-n_i) , \mathfrak{s}|_{X(-n_i)}) = -\frac{c_i^2}{8n_i} + \frac{1}{8}.
\]
Similarly, since $\widetilde{\mathfrak{s}}$ is sharp, $\delta( \widetilde{Y}_i , \widetilde{\mathfrak{s}}_0 |_{\widetilde{Y}_i}) = -\widetilde{c}_i^2/(8(n_i/p)) +1/8$. 

Let $\mathfrak{s}_{can}$ denote the canonical spin$^c$-structure on $X(-n_i)$. Then $\langle c(\mathfrak{s}_{can}) , [S_i] \rangle = 2-n_i$. Hence $\mathfrak{s}|_{X(-n_i)} = \mathcal{O}(l_i) \otimes \mathfrak{s}_{can}$, where $c_i = 2l_i + 2 - n_i$ (so $l_i = (c_i + n_i)/2 - 1$). The pullback of $\mathfrak{s}_{can}$ under the covering $\widetilde{Y}_i \to Y_i$ is the canonical spin$^c$-structure on $\widetilde{Y}_i$, hence $\mathcal{O}(l_i) \otimes \mathfrak{s}_{can} |_{Y_i}$ pulls back to $\mathcal{O}(l_i) \otimes \mathfrak{s}_{can} |_{\widetilde{Y}_i}$. This extends over $X(-n_i/p)$ as $\mathcal{O}(l_i) \otimes \mathfrak{s}_{can}$, however this may not be a sharp extension. From Lemma \ref{lem:equiv}, we must have $\widetilde{\mathfrak{s}}|_{\widetilde{Y}_i} = \mathcal{O}(l_i - (n_i/p)u_i) \otimes \mathfrak{s}_{can}$ for some $u_i \in \mathbb{Z}$ (because $[S]|_{X(-n_i/p)} \cong \mathcal{O}(n_i/p)$). Therefore,
\begin{align*}
\widetilde{c}_i &= 2l_i - 2 (n_i/p) u_i + 2 - (n_i/p) \\
&= c_i + n_i - 2 (n_i/p) u_i - (n_i/p).
\end{align*}
Hence $ \dfrac{(c_i+n_i)}{2(n_i/p)} - \dfrac{(\widetilde{c}_i + (n_i/p))}{2(n_i/p)} = u_i \in \mathbb{Z}$, as claimed. The requirement that $-n_i/p < \widetilde{c}_i \le n_i/p$ uniquely determines $u_i$ and hence uniquely determined $\widetilde{c}_i$ from $c_i,n_i,p$.
\end{proof}

\begin{lemma}\label{lem:dimnu}
Assume that $\mathfrak{s}, \widetilde{\mathfrak{s}}$ are sharp. Then
\[
d(\widetilde{X}_0 , \widetilde{\mathfrak{s}}_0) = p \, d(X_0 , \mathfrak{s}_0) + \sum_i \nu_i
\]
where
\begin{equation}\label{equ:nu}
\nu_i = (p-1) + \frac{1}{4(n_i/p)}\left( (c_i^2 - n_i^2) - ( \widetilde{c}_i^2 - (n_i/p)^2 ) \right).
\end{equation}
\end{lemma}
\begin{proof}
From Proposition \ref{prop:ubcov}, we have 
\[
d(\widetilde{X}_0 , \widetilde{\mathfrak{s}}_0) = p \, d(X_0 , \mathfrak{s}_0) + \sum_i \nu_i,
\]
where
\[
\nu_i = \frac{(p-1)}{2} + 2( \delta(\widetilde{Y}_i , \widetilde{\mathfrak{s}}|_{\widetilde{Y}_i}) - p \delta( Y_i , \mathfrak{s}|_{Y_i}) ) - \frac{3}{4}\rho(Y_i).
\]
Since $Y_i = L(n_i,1)$, $\widetilde{Y}_i = L(n_i/p,1)$, we have
\begin{align*}
\rho(Y_i) &= \eta_{sig}(\widetilde{Y}_i) - p \eta_{sig}(Y_i) \\
&= -4 s(1, n_i/p) + 4p s(1,n_i) \\
&= -\frac{(n_i/p-1)(n_i/p-2)}{3(n_i/p)} + \frac{(n_i-1)(n_i-2)}{3(n_i/p)} \\
&= \frac{(p^2-1)n_i}{3p} - (p-1)
\end{align*}
where the eta invariants are calculated using the quotient of the round metric on $S^3$.

From Lemma \ref{lem:deltaY}, we have $\delta( Y_i , \mathfrak{s}|_{Y_i}) = -c_i^2/(8n_i) + 1/8$, $\delta( \widetilde{Y}_i , \widetilde{\mathfrak{s}} |_{\widetilde{Y}_i}) = -\widetilde{c}_i^2/(8(n_i/p)) + 1/8$. Then a straightforward calculation gives:
\[
\nu_i = (p-1) + \frac{1}{4(n_i/p)}\left( (c_i^2 - n_i^2) - ( \widetilde{c}_i^2 - (n_i/p)^2 ) \right).
\]
\end{proof}

\begin{lemma}\label{lem:nu1}
Let $\nu_i$ be defined as in Equation (\ref{equ:nu}). Then
\begin{itemize}
\item[(1)]{$\nu_i = p-1$ for $c_i = n_i$, $\widetilde{c_i} = n_i/p$.}
\item[(2)]{$\nu_i = 0$ if and only if $c_i = \pm( n_i - 2)$, $\widetilde{c_i} = \pm ( n_i/p - 2)$.}
\item[(3)]{$\nu_i < 0$ otherwise.}
\end{itemize}

\end{lemma}
\begin{proof}
To simplify notation, we will fix the value of $i$ and omit subscripts. Then 
\[
\nu = p-1 + \frac{1}{4(n/p)}\left(  (c^2 - n^2) - (\widetilde{c}^2 - (n/p)^2 ) \right).
\]
Define $x,y \in \mathbb{Q}$ by
\[
x = \frac{n - c}{2(n/p)}, \quad y = \frac{ n/p - \widetilde{c} }{2(n/p)}.
\]
Note that $c = n \; ({\rm mod} \; 2)$, $\widetilde{c} = n/2 \; ({\rm mod} \; 2)$, so $x,y \in \frac{1}{n/p} \mathbb{Z}$. 

Since $-n < c \le n$, $-n/p < \widetilde{c} \le n/p$, we have $0 \le x < p$ and $0 \le y < 1$. Furthermore, $x-y \in \mathbb{Z}$ by Lemma \ref{lem:deltaY} (2). Write $x = v + y$. Then $v \in \mathbb{Z}$ and $y$ is the fractional part of $x$. We also have $0 \le v \le p-1$. In terms of $x,y$, we have
\begin{align*}
\nu &= p-1 + \frac{n}{p}\left( x(x-p) - y(y-1) \right) \\
&= p-1 + \frac{n}{p}\left( (2v-(p-1))y + v(v-p) \right).
\end{align*}
If $v=0$, then
\[
\nu = (p-1)\left( 1 - \frac{n}{p}y \right).
\]
This equals $p-1$ for $y=0$, $0$ for $y = 1/(n/p)$ and is negative otherwise.

Now suppose that $v \neq 0$. Then $1 \le v \le p-1$ and hence $v(v-p) \le -(p-1)$ with equality if and only if $v=1$ or $p-1$. Using this and $v \le p-1$, we get
\[
\nu \le p-1 + \frac{n}{p}\left( (p-1)y - (p-1) \right).
\]
Furthermore, $y < 1$ and $(n/p)y \in \mathbb{Z}$, so $y \le 1 - 1/(n/p)$, giving
\[
\nu \le 0.
\]
Equality occurs if and only if $y = 1 - 1/(n/p)$ and $v = p-1$. So we have shown that $\nu = p-1$ if and only if $v=y=0$ (in which case $c=n$, $\widetilde{c} = n/p$), $\nu = 0$ if and only if $(v,y) = (0, 1/(n/p))$ or $(v,y) = (p-1 , 1-1/(n/p))$ (in which case $c = \pm(n-2)$, $\widetilde{c} = \pm (n/p-2)$) and $\nu < 0$ otherwise. 

\end{proof}

\begin{lemma}\label{lem:nu2}
Let $\nu_i$ be defined as in Equation (\ref{equ:nu}). Then
\[
\nu_i \ge (p-1) + \frac{(p-1)}{2}( c_i - n_i ).
\]
\end{lemma}
\begin{proof}
As in the proof of Lemma \ref{lem:nu1}, we omit subscripts. We have
\begin{align*}
\nu - (p-1) - \frac{(p-1)}{2}(c-n) &= \frac{1}{4(n/p)}\left( c^2 - n^2 - \widetilde{c}^2 + (n/p)^2 \right) - \frac{(p-1)}{2}(c-n) \\
&= \frac{1}{4(n/p)}\left( c^2 - n^2 - \widetilde{c}^2 + (n/p)^2 - 2(n/p)(p-1)(c-n) \right) \\
&= \frac{1}{4(n/p)}\left( c^2 - n^2 - \widetilde{c}^2 + (n/p)^2 - 2n(c-n) + 2(n/p)(c-n) \right) \\
&= \frac{1}{4(n/p)}\left( (c-n)^2 + 2(n/p)(c-n) - \widetilde{c}^2 + (n/p)^2 \right) \\
&= \frac{1}{4(n/p)}\left( (c-n)(c-n + 2n/p) - \widetilde{c}^2 + (n/p)^2 \right) \\
&= \frac{1}{4(n/p)}\left( (n-c)(n-c - 2n/p) - \widetilde{c}^2 + (n/p)^2 \right).
\end{align*}
Now if $c \le n - 2n/p$, then $(n-c)(n-c - 2n/p) \ge 0$. Also $ - \widetilde{c}^2 + (n/p)^2 \ge 0$ because $|\widetilde{c}| \le n/p$, hence in this case we get $\nu - (p-1) - \frac{(p-1)}{2}(c-n) \ge 0$. Now suppose $n - 2n/p \le c \le n$. Then it follows that $\widetilde{c} = c - (p-1)n/p$ and a direct calculation shows that $\nu - (p-1) - \frac{(p-1)}{2}(c-n) = 0$. So in either case we have $\nu_i \ge (p-1) + \frac{(p-1)}{2}(c-n)$.
\end{proof}

\begin{theorem}\label{thm:1sad}
Let $X$ be a compact, oriented, smooth $4$-manifold with $b_1(X) = 0$, $b_+(X) > 1$ of simple type. Assume that $SW(X , \mathfrak{s}) \neq 0 \; ({\rm mod} \; p)$ for some prime $p$. Suppose that $H_1(X ; \mathbb{Z}_p) = 0$. Let $S \subset X$ be an embedded surface with $[S] \in H^2(X ; \mathbb{Z})$ non-torsion and divisible by $p$. Then $| \langle c(\mathfrak{s}) , [S] \rangle | + [S]^2 < 0$.
\end{theorem}
\begin{proof}
The conditions on $X$ and $S$ ensure that a degree $p$ cyclic branched cover $\widetilde{X} \to X$ exists with branch locus $S$. Since the branch locus consists of a single sphere we can take the normal weight to be $\phi = 1$. Since $[S]$ is non-torsion, the adjunction inequality implies that $[S]^2 < 0$ (see Section \ref{sec:config}). Let $n = -[S]^2$. Since $X$ has simple type and $SW(X , \mathfrak{s}) \neq 0$, then $d(X , \mathfrak{s}) = 0$ and $\mathfrak{s}$ is sharp. After possibly changing the orientation of $S$, we can assume that $c \ge 0$, where $c = \langle c(\mathfrak{s}) , [S] \rangle$. Let $\mathfrak{s}_0 = \mathfrak{s}|_{X_0}$, $\widetilde{\mathfrak{s}}_0 = \pi^*( \mathfrak{s}_0)$ and let $\widetilde{\mathfrak{s}}$ be a sharp extension of $\widetilde{\mathfrak{s}}_0$. Let $\widetilde{c} = \langle c(\widetilde{\mathfrak{s}}) , [\widetilde{S}] \rangle$. Then $| \widetilde{c} | \le n/p$ and we may choose $\widetilde{\mathfrak{s}}$ so that $\widetilde{c} \neq -n/p$.

Sharpness of $\mathfrak{s}$ says that $|\langle c(\mathfrak{s}) , [S] \rangle| + [S]^2 \le 0$. Suppose that $\langle c(\mathfrak{s}) , [S] \rangle | + [S]^2 = 0$. Equivalently, $c=n$. Therefore $d(\widetilde{X} , \widetilde{\mathfrak{s}}) = p-1$ by Lemma \ref{lem:dimnu}. In particular, $d(\widetilde{X} , \widetilde{\mathfrak{s}}) > 0$. Since $c = n$, we can take $\alpha_{-j} = j\alpha = (j/p)[S]$. If $SW(X , \alpha_{-j} + \mathfrak{s}) \neq 0 \; ({\rm mod} \; p)$, then by the simple type assumption on $X$, we must have $d(X , \alpha_{-j} + \mathfrak{s}) = d(X_0 , \mathfrak{s}_{-j}) = 0$. Since
\[
d(X_0 , \mathfrak{s}_{-j}) = 2 ind_{APS}(X_0 , \mathfrak{s}_{-j}) - b_+(X_0) - 1,
\]
it follows that
\begin{align*}
d(X_0 , \mathfrak{s}_{-j}) - d(X_0 , \mathfrak{s}_0) &= 2( ind_{APS}(X_0 , \mathfrak{s}_{-j}) - ind_{APS}(X_0 , \mathfrak{s}_0) ) \\
&= 2( \delta( Y , \mathfrak{s}_{-j} |_Y ) - \delta( Y , \mathfrak{s}_0 |_Y ) ) \\
&= \frac{1}{4n}\left( c^2 - (c')^2  \right) \\
&= \frac{1}{4n}\left( n^2 - (c')^2 \right)
\end{align*}
where $c' = \langle c(\mathfrak{s}_{-j}) , [S] \rangle = \langle 2 (j/p)[S] + c(\mathfrak{s}) , [S] \rangle = c - 2(j/p)n = n - 2(j/p)n$. Hence $d(X_0 , \mathfrak{s}_{-j}) = 0$ if and only if $c' = \pm n$. But since $c' = n - 2(j/p)n$, this can only happen if $j = 0 \; ({\rm mod} \; p)$. So of the spin$^c$-structures $\{ \mathfrak{s}_j \}_{j=0}^{p-1}$, the only one with $SW(X , \mathfrak{s}_j ) \neq 0 \; ({\rm mod} \; p)$ is $\mathfrak{s}_0$. Proposition \ref{prop:dimbound} then gives $d(\widetilde{X}_0 , \mathfrak{s}_0) \le 0$. This is a contradiction, since $d(\widetilde{X}_0 , \mathfrak{s}_0) = d(\widetilde{X} , \widetilde{\mathfrak{s}}) = p-1 > 0$. So we must have that $\langle c(\mathfrak{s}) , [S] \rangle | + [S]^2 < 0$.

\end{proof}

We now consider the case of multiple spheres. We obtain a result similar to Theorem \ref{thm:1sad}.

\begin{theorem}
Let $X$ be a compact, oriented, smooth $4$-manifold with $b_1(X) = 0$, $b_+(X) > 1$ of simple type. Assume that $SW(X , \mathfrak{s}) \neq 0 \; ({\rm mod} \; p)$ for some prime $p$. Suppose that $H_1(X ; \mathbb{Z}_p) = 0$. Suppose that $S_1, \dots , S_r \subset X$ are disjoint embedded spheres with negative self-intersection whose Poincar\'e dual classes are linearly depenent in $H^2(X ; \mathbb{Z}_p)$, but any $r-1$ of them are linearly independent in $H^2(X ; \mathbb{Z}_p)$. Then 
\[
\sum_i | \langle c(\mathfrak{s}) , [S_i] \rangle | + [S_i]^2 \le 4 - 2r.
\]
\end{theorem}
\begin{proof}
The conditions on $X$ and $S_1, \dots , S_r$ ensure that a degree $p$ cyclic branched cover $\widetilde{X} \to X$ exists with branch locus $S_1, \dots , S_r$. Furthermore

Since $X$ has simple type and $SW(X , \mathfrak{s}) \neq 0$, then $d(X , \mathfrak{s}) = 0$ and $\mathfrak{s}$ is sharp. After possibly changing the orientation of each $S_i$, we can assume that $c_i \ge 0$, where $c_i = \langle c(\mathfrak{s}) , [S_i] \rangle$. Let $\mathfrak{s}_0 = \mathfrak{s}|_{X_0}$, $\widetilde{\mathfrak{s}}_0 = \pi^*( \mathfrak{s}_0)$ and let $\widetilde{\mathfrak{s}}$ be a sharp extension of $\widetilde{\mathfrak{s}}_0$. Let $\widetilde{c}_i = \langle c(\widetilde{\mathfrak{s}}) , [\widetilde{S}_i] \rangle$. Then $| \widetilde{c}_i | \le n_i/p$ and we may choose $\widetilde{\mathfrak{s}}$ so that $\widetilde{c}_i \neq -n_i/p$.

Since $\mathfrak{s}, \widetilde{\mathfrak{s}}$ are sharp, Lemma \ref{lem:dimnu} gives $d(\widetilde{X} , \widetilde{\mathfrak{s}}) = \sum_i \nu_i$, where
\[
\nu_i = (p-1) + \frac{1}{4(n_i/p)}\left( (c_i^2 - n_i^2) - ( \widetilde{c}_i^2 - (n_i/p)^2 ) \right).
\]
Then, Lemma \ref{lem:nu2} gives
\begin{align*}
d(\widetilde{X} , \widetilde{\mathfrak{s}}) &= \sum_i \nu_i \\
& \ge r(p-1) + \frac{(p-1)}{2} \sum_i (c_i - n_i ) \\
&= r(p-1) + \frac{(p-1)}{2}\left( \sum_i \langle c(\mathfrak{s}) , [S_i] \rangle + [S_i]^2 \right).
\end{align*}
On the other hand, Proposition \ref{prop:dimbound} gives $d(\widetilde{X} , \widetilde{\mathfrak{s}}) \le 2(p-1)$, hence
\[
2(p-1) \ge d(\widetilde{X} , \widetilde{\mathfrak{s}}) \ge  r(p-1) + \frac{(p-1)}{2}\left( \sum_i \langle c(\mathfrak{s}) , [S_i] \rangle + [S_i]^2 \right),
\]
which gives $\sum_i \langle c(\mathfrak{s}) , [S_i] \rangle + [S_i]^2 \le 4-2r$.
\end{proof}

\subsection{Double covers}

In the case of double covers ($p=2$), the general results simplify considerably and allows us to obtain more precise results.

\begin{theorem}\label{thm:p=2}
Let $X$ be a compact, oriented, smooth $4$-manifold with $b_+(X) > 1$ and $H_1(X ; \mathbb{Z}_2) = 0$. Suppose $X$ has simple type. Suppose $S_1, \dots , S_r \subset X$ are disjoint smoothly embedded spheres with $[S_i]^2 < 0$ for all $i$. Suppose that $[S] = [S_1] + \cdots + [S_r]$ is divisible by $2$ and that any $r-1$ of $[S_1], \dots , [S_r]$ are linearly independent in $H^2(X ; \mathbb{Z})$. Suppose $\mathfrak{s}$ is a spin$^c$-structure on $X$ with $SW(X , \mathfrak{s}) = 1 \; ({\rm mod} \; 2)$ and orient $S_1, \dots , S_r$ such that $\langle c(\mathfrak{s}) , [S_i] \rangle \ge 0$ for all $i$. Then the following holds:
\begin{itemize}
\item[(1)]{$ \langle c(\mathfrak{s}) , [S] \rangle + [S]^2 \le 4 - 2r.$}
\item[(2)]{If $\langle c(\mathfrak{s}) , [S] \rangle + [S]^2 = -2r$, then $d(\widetilde{X} , \widetilde{\mathfrak{s}}) = 0$ and
\[
SW(\widetilde{X} , \widetilde{\mathfrak{s}}) = SW(X , \mathfrak{s}) + SW(X , [S]/2 \otimes \mathfrak{s}) \; ({\rm mod} \; 2).
\]
}
\item[(3)]{If $\langle c(\mathfrak{s}) , [S] \rangle + [S]^2 = 4-2r$, then $r \ge 4$, $[S]^2 = 8-4r$, $\langle c(\mathfrak{s}) , [S] \rangle = 2r-4$, $d(\widetilde{X} , \widetilde{\mathfrak{s}}) = 2$ and
\[
SW(\widetilde{X} , \widetilde{\mathfrak{s}}) = SW(X , \mathfrak{s}) = SW(X , [S]/2 \otimes \mathfrak{s}) = 1 \; ({\rm mod} \; 2).
\]
In particular, $\widetilde{X}$ is not simple type.
}
\item[(4)]{If $\langle c(\mathfrak{s}) , [S] \rangle + [S]^2 < -2r$, then $SW(\widetilde{X} , \widetilde{\mathfrak{s}}) = 0$.}
\end{itemize}
\end{theorem}
\begin{proof}
By Lemma \ref{lem:dimnu}, $d(\widetilde{X}_0 , \widetilde{\mathfrak{s}}_0) = \sum_i \nu_i$, where
\[
\nu_i = 1 + \frac{1}{2n_i}\left( (c_i^2 - n_i^2) - ( \widetilde{c}_i^2 - (n_i/2)^2 ) \right).
\]
Since $0 \le c_i \le n_i$ and since $\widetilde{c}_i = c_i + (n_i/2) + 2n_i u_i$ for some integer $u_i$ (by Lemma \ref{lem:deltaY}), it can be assumed that $\widetilde{c}_i = c_i - n_i/2$. Indeed, this satisfies $-n_i/2 < c_i - n_i/2 \le n_i/2$. It follows that
\[
\nu_i = 1 + \frac{c_i - n_i}{2}.
\]
Summing over $i$, we get
\[
d(\widetilde{X} , \widetilde{\mathfrak{s}}) = r + \frac{1}{2}\left( \langle c(\mathfrak{s}) , [S] \rangle + [S]^2 \right).
\]
Since $SW(X , \mathfrak{s}) \neq 0 \; ({\rm mod} \; p)$, Proposition \ref{prop:dimbound} gives $d(\widetilde{X} , \widetilde{\mathfrak{s}}) \le 2$ and hence $\langle c(\mathfrak{s}) , [S] \rangle + [S]^2 \le 4-2r$, which proves (1). If $\langle c(\mathfrak{s}) , [S] \rangle + [S]^2 = -2r$, then $d(\widetilde{X} , \widetilde{\mathfrak{s}}) = 0$ and Theorem \ref{thm:swcov} gives
\[
SW(\widetilde{X} , \widetilde{\mathfrak{s}} ) = SW(X , \mathfrak{s}) + SW(X , [S]/2 + \mathfrak{s}) \; ({\rm mod} \; 2),
\]
which proves (2).

If $\langle c(\mathfrak{s}) , [S] \rangle + [S]^2 = 4-2r$, then $d(\widetilde{X} , \widetilde{\mathfrak{s}}) = 2$. Proposition \ref{prop:dimbound} (1) implies that $SW(X , [S]/2 + \mathfrak{s}) \neq 0 \; ({\rm mod} \; 2)$ and the Theorem \ref{thm:swcov} implies that
\[
SW(\widetilde{X} , \widetilde{\mathfrak{s}}) = SW(X , \mathfrak{s}) = SW(X , [S]/2 + \mathfrak{s}) = 1 \; ({\rm mod} \; 2).
\]
Hence $\widetilde{X}$ is not of simple type. Furthermore, since $SW(X , [S]/2 + \mathfrak{s}) \neq 0$ and $X$ is of simple type, $d(X , [S]/2 + \mathfrak{s}) = d(X , \mathfrak{s}) = 0$. But
\begin{align*}
d(X , [S]/2 + \mathfrak{s}) - d(X , \mathfrak{s}) &= \frac{ ( [S] + c(\mathfrak{s}) )^2 - c(\mathfrak{s})^2 }{4} \\
&= \frac{ 2\langle c(\mathfrak{s}) , [S] \rangle + [S]^2 }{4}.
\end{align*}
Hence $2 \langle c(\mathfrak{s}) , [S] \rangle + [S]^2 = 0$. But we also have $\langle c(\mathfrak{s}) , [S] \rangle + [S]^2 = 4-2r$. Hence $\langle c(\mathfrak{s}) , [S] \rangle = 2r-4$, $[S]^2 = 8-4r$. Now since $n_i$ is even for all $i$, $[S_i]^2 \le -2$ and hence $8-4r = [S]^2 \le -2r$, giving $r \ge 4$. This proves (3). Lastly, if $\langle c(\mathfrak{s}) , [S] \rangle + [S]^2 < -2r$, then $d(\widetilde{X} , \widetilde{\mathfrak{s}}) < 0$, hence $SW(\widetilde{X} , \widetilde{\mathfrak{s}}) = 0$, which proves (4).
\end{proof}

We state as a corollary the $r=1$ case of Theorem \ref{thm:p=2} in the most interesting case where $\langle c(\mathfrak{s}) , [S] \rangle + [S]^2 = -2$.

\begin{corollary}
Let $X$ be a compact, oriented, smooth $4$-manifold with $b_+(X) > 1$ and $H_1(X ; \mathbb{Z}_2) = 0$. Suppose $X$ has simple type. Let $S \subset X$ be an embedded sphere and suppose $[S]$ is divisible by $2$. Let $\mathfrak{s}$ be a spin$^c$-structure on $X$ and suppose that $\langle c(\mathfrak{s}) , [S] \rangle + [S]^2 = -2$.
\begin{itemize}
\item[(1)]{If $[S]^2 = -4$, then $SW(\widetilde{X} , \widetilde{\mathfrak{s}}) = SW(X , \mathfrak{s}) + SW(X , [S]/2 \otimes \mathfrak{s}) \; ({\rm mod} \; 2)$.}
\item[(2)]{If $[S]^2 < -4$, then $SW(\widetilde{X} , \widetilde{\mathfrak{s}}) = SW(X , \mathfrak{s}) \; ({\rm mod} \; 2)$.}
\end{itemize}
\end{corollary}
\begin{proof}
(1) follows from the proof of Theorem \ref{thm:p=2}, or alternatively from Theorem \ref{thm:swcov}. For (2), the same calculation given in the proof of Theorem \ref{thm:p=2}, gives
\[
d(X , [S]/2 + \mathfrak{s}) - d(X , \mathfrak{s}) =  \frac{ 2\langle c(\mathfrak{s}) , [S] \rangle + [S]^2 }{4} = \frac{-4 - [S]^2}{4},
\]
where we used that $\langle c(\mathfrak{s}) , [S] \rangle + [S]^2 = -2$. Note that that assumption also gives $d(X , \mathfrak{s}) = 0$, so $d(X , [S]/2 + \mathfrak{s}) = 0$ if and only if $[S]^2 = -4$. But we assumed that $X$ is simple type, so if $[S]^2 < -4$, then $d(X , [S]/2 + \mathfrak{s}) \neq 0$, so that $SW(X , [S]/2 + \mathfrak{s}) = 0$ which proves (2).
\end{proof}

\section{Double covers and embedded projective planes}\label{sec:rp2}

Let $X$ be a compact, oriented, smooth $4$-manifold with $b_+(X) > 1$ and $H_1(X ; \mathbb{Z}_2) = 0$. Suppose $S_1, \dots , S_r \subset X$ are disjoint smoothly embedded projective planes with Euler numbers $n_i = e(S_i)$. Let $X_0$ be $X$ with a tubular neighbourhood of $S = S_1 \cup \cdots \cup S_r$ removed. Then $X_0$ has ingoing boundary $Y = Y_1 \cup \cdots \cup Y_r$, $Y_i = Y(n_i)$.

\begin{lemma}\label{lem:rpcov}
Suppose that any $r-1$ of $[S_1] , \dots , [S_r]$ are linearly independent in $H_2(X ; \mathbb{Z}_2)$ and that $[S_1] + \cdots + [S_r] = 0 \in H_2(X ; \mathbb{Z}_2)$. Then $H_1(X_0 ; \mathbb{Z}) \cong \mathbb{Z}_2$. Hence there is a uniquely determined non-trivial double cover $\pi : \widetilde{X}_0 \to X_0$. Furthermore, $b_1(\widetilde{X}_0) = 0$ and writing the ingoing boundary of $\widetilde{X}_0$ as $\widetilde{Y} = \widetilde{Y}_1 \cup \cdots \cup \widetilde{Y}_r$, we have that $\widetilde{Y}_i \cong Y(n_i/2)$ and $\widetilde{Y}_i \to Y_i$ is the non-trivial double cover $Y(n_i/2) \to Y(n_i)$.
\end{lemma}
\begin{proof}
By Excision and Poincar\'e--Lefschetz, we have
\[
H_2(X,X_0 ; \mathbb{Z}_2) \cong \bigoplus_i H_2(R(n_i), Y(n_i) ; \mathbb{Z}_2) \cong \bigoplus_i H^2(R(n_i) ; \mathbb{Z}_2) \cong \bigoplus_i \mathbb{Z}_2.
\]
Then the long exact sequence for $(X , X_0)$ gives
\[
\cdots \to H_2(X ; \mathbb{Z}_2) \buildrel \varphi \over \longrightarrow \bigoplus_i \mathbb{Z}_2 \buildrel \partial \over \longrightarrow H_1(X_0 ; \mathbb{Z}_2) \to H_1(X ; \mathbb{Z}_2) \to 0.
\]
The existence of the double cover $\widetilde{X}_0 \to X_0$ which restricts to $Y(n_i/2) \to Y(n_i)$ on each boundary component now follows from the same argument used in Lemma \ref{lem:bre}. That $b_1(\widetilde{X}_0) = 0$ follows from the proofs of Lemmas \ref{lem:b10} and \ref{lem:1p}.
\end{proof}

Suppose $S_1, \dots , S_r$ are as in Lemma \ref{lem:rpcov}. By attaching a copy of $R(n_i/2)$ to each boundary component we can obtain a smooth closed $4$-manifold $\widetilde{X}$ which is a branched double cover of $X$, branched over $S_1, \dots, S_r$. However it turns out that $\widetilde{X}$ does not behave well with respect to spin$^c$-structures. The reason for this is as follows. Suppose $\mathfrak{s}$ is a spin$^c$-structure on $X$ and let $\mathfrak{s}_0 = \mathfrak{s}|_{X_0}$ be its restriction to $X_0$. Let $\widetilde{\mathfrak{s}}_0 = \pi^*(\mathfrak{s}_0)$ be the pullback of $\mathfrak{s}_0$ to $\widetilde{X}_0$. Then $\widetilde{\mathfrak{s}}$ does not extend to a spin$^c$-structure on $\widetilde{X}$. This is because any spin$^c$-structure on $Y(n/2)$ which is a pullback under the cover $Y(n/2) \to Y(n)$ does not extend to $R(n/2)$ (which follows from Lemma \ref{lem:pullbackyn}). Because of this we are not able to obtain a relationship between the Seiberg--Witten invariants of $X$ and $\widetilde{X}$. However, if we assume that $n_i \ge -2$ for all $i$, then we can obtain a closed $4$-manifold $\widehat{X}$ as follows. If $n_i > 0$, then we attach to the $i$-th boundary component of $\widetilde{X}_0$ the $D_{n_i/2+2}$ plumbing $P(D_{n_i/2+2})$. If $n_i = -2$, then we attach a copy of $X(-4)$ (which has boundary $L(4,1) = Y(-1)$) and if $n_i = 0$, then we attach the boundary connected sum of two copies of $X(-2)$ (which has boundary $L(2,1) \# L(2,1) = Y(0)$). Every spin$^c$-structure $\widetilde{\mathfrak{s}}_0$ on $\widetilde{X}_0$ has an extension $\widehat{\mathfrak{s}}$ to $\widehat{X}$ such that $d(\widehat{X} , \widehat{\mathfrak{s}}) = d(\widetilde{X}_0 , \widetilde{\mathfrak{s}}_0)$. Call an extension $\widehat{\mathfrak{s}}$ of $\widetilde{\mathfrak{s}}_0$ {\em sharp} if $d(\widehat{X} , \widehat{\mathfrak{s}}) = d(\widetilde{X}_0 , \widetilde{\mathfrak{s}}_0)$.

If $n_i \ge 0$ for all $i$, then we can also form a closed $4$-manifold $X'$ from $X_0$ in a similar manner. If $n_i > 0$, we attach a copy of $P(D_{n_i+2})$. If $n_i = 0$, we attach the boundary connected sum of two copies of $X(-2)$. Every spin$^c$-structure on $X_0$ has a sharp extension to $X'$.

Let $\mathfrak{s}$ be a spin$^c$-structure on $X$, $\mathfrak{s}_0$ its restriction to $X_0$. Let $A \in H^2(X ; \mathbb{Z}_2)$ be the class of the flat line bundle defined by the double cover $\widetilde{X}_0 \to X_0$. Let $\widetilde{\mathfrak{s}}_0 = \pi^*(\mathfrak{s}_0)$.

\begin{lemma}\label{lem:rpcov}
Let $n = n_1 + \cdots + n_r$. There exists $\varepsilon \in \mathbb{Z}$, $| \varepsilon | \le r$, $\varepsilon = n/2 \; ({\rm mod} \; 4)$ such that the following hold:
\begin{align*}
\sigma(\widetilde{X}_0) &= 2 \sigma(X_0) - \frac{n}{2} \\
b_+(\widetilde{X}_0) &= 2b_+(X_0) + 1 -\frac{r}{2} - \frac{n}{4} \\
ind_{APS}(\widetilde{X}_0 , \widetilde{\mathfrak{s}}_0) &= ind_{APS}(X_0 , \mathfrak{s}_0) + \frac{n}{8} - \frac{\varepsilon}{4} \\
d(\widetilde{X}_0 , \widetilde{\mathfrak{s}}_0) &= 2d(X_0 , \mathfrak{s}_0) + \frac{1}{2}( r+n-\varepsilon) \\
d(X_0 , A \otimes \mathfrak{s}_0) &= d(X_0 , \mathfrak{s}_0) + \frac{n}{4} - \frac{\varepsilon}{2}.
\end{align*}

\end{lemma}
\begin{proof}
Let $m$ be even. We compute the $\rho$-invariant for the covering $Y(m/2) \to Y(m)$. If $m \neq 0$, then $Y(m)$ is spherical and we use the psc metric induced from the round metric on $S^3$. From Section \ref{sec:prism}, we have $\eta_{sig}(Y(m)) = \frac{1}{6m}(2m^2+1)$ for $m>0$. The same formula holds for $m<0$ because $\eta_{sig}(Y(m)) = -\eta_{sig}(Y(-m)) = \frac{1}{6m}(2m^2+1)$. Then
\begin{align*}
\rho(Y(m)) &= \eta_{sig}(Y(m/2)) - 2\eta_{sig}(Y(m)) \\
&= \frac{1}{3m}(\frac{1}{2}m^2 + 1) - \frac{1}{3m}(2m^2+1) \\
&= -\frac{m}{2}.
\end{align*}

For $m=0$, first note that $Y(0) = (S^1 \times S^2) / \sigma$ where $\sigma$ acts as a reflection on $S^1$ and as the antipodal map on $S^2$. Give $S^1 \times S^2$ the product of the round metrics on $S^1$ and $S^2$. This metric is preserved by $\sigma$ and we give $Y(0)$ the induced quotient metric $g$. The double covering $\pi : Y(0) \to Y(0)$ is the quotient under $\sigma$ of the double covering $S^1 \times S^2 \to S^1 \times S^2$ given by $(x,y) \mapsto (x^2 , y)$, where we view $S^1$ as the unit circle in $\mathbb{C}$. The antipodal map on $S^2$ induces an orientation reversing diffeomorphism of $Y(0)$ which preserves both $g$ and $\pi^*(g)$. Hence $\eta_{sig}(Y(0) , g) = \eta_{sig}(Y(0) , \pi^*(g)) = 0$ and hence $\rho(Y(0)) = 0$. This shows that $\rho(Y(m)) = -m/2$ for all $m$.

Let $Y_i = Y(n_i)$ be the $i$-th boundary component. Since $\mathfrak{s}|_{Y_i}$ extends over $R(n_i)$, we have that $\mathfrak{s}|_{Y_i} = \mathfrak{s}_{0,1}$ or $\mathfrak{s}_{1,1}$. We let $\varepsilon_i = 1$ if $\mathfrak{s}|_{Y_i} = \mathfrak{s}_{1,1}$ and $\varepsilon_i = -1$ if $\mathfrak{s}|_{Y_i} = \mathfrak{s}_{0,1}$. From Lemma \ref{lem:pullbackyn}, we have that $\widetilde{\mathfrak{s}}_0 |_{\widetilde{Y}_i}$ is $\mathfrak{s}_{1,0}$ if $\varepsilon_i = 1$ and is $\mathfrak{s}_{0,0}$ if $\varepsilon_i = -1$. Hence from the computation of the delta invariants in Section \ref{sec:prism}, we have $\delta( \widetilde{Y}_i , \widetilde{\mathfrak{s}}_0 |_{\widetilde{Y}_i}) = n/16 - \varepsilon_i/4$. Hence
\[
\delta( \widetilde{Y}_i , \widetilde{\mathfrak{s}}_0 |_{\widetilde{Y}_i}) - 2\delta(Y_i , \mathfrak{s}|_{Y_i}) = \frac{n}{16} - \frac{\varepsilon_i}{4}.
\]
Set $\varepsilon = \varepsilon_1 + \cdots + \varepsilon_r$. So $|\varepsilon| \le r$. Then from Proposition \ref{prop:ubcov}, we have
\begin{align*}
\sigma(\widetilde{X}_0) &= 2 \sigma(X_0) - \frac{n}{2} \\
b_+(\widetilde{X}_0) &= 2b_+(X_0) + 1 -\frac{r}{2} - \frac{n}{4} \\
ind_{APS}(\widetilde{X}_0 , \widetilde{\mathfrak{s}}_0) &= 2 ind_{APS}(X_0 , \mathfrak{s}_0) + \frac{n}{8} - \frac{\varepsilon}{4} \\
d(\widetilde{X}_0 , \widetilde{\mathfrak{s}}_0) &= 2d(X_0 , \mathfrak{s}_0) + \frac{1}{2}( r+n-\varepsilon).
\end{align*}
Integrality of $ind_{APS}(\widetilde{X}_0 , \widetilde{\mathfrak{s}}_0)$ gives $\varepsilon = n/2 \; ({\rm mod} \; 4)$.

Lastly, since $d(X_0 , \mathfrak{s}_{0}) = 2 ind_{APS}(X_0 , \mathfrak{s}_{0}) - b_+(X_0) - 1$, $d(X_0 , A \otimes \mathfrak{s}_0) = 2ind_{APS}(X_0 , A \otimes \mathfrak{s}_0) - b_+(X_0) - 1$, we have
\[
d(X_0 , A \otimes \mathfrak{s}_0) - d(X_0 , \mathfrak{s}_0) = 2( \delta( Y , A \otimes \mathfrak{s}_0 |_Y ) - \delta( Y , \mathfrak{s}_0 |_Y)) = 2 \delta(Y , A \otimes \mathfrak{s}_0 |_Y).
\]
Hence $d(X_0 , A \otimes \mathfrak{s}_0) = d(X_0 , \mathfrak{s}_0) + 2 \delta(Y , A\otimes \mathfrak{s}_0 |_Y)$. Now if $\mathfrak{s}_0|_{Y_i} = \mathfrak{s}_{u,1}$, then $A \otimes \mathfrak{s}_0 |_{Y_i} = \mathfrak{s}_{u,0}$. It follows that $\delta(Y , A \otimes \mathfrak{s}_0|_Y) = (n - 2 \varepsilon)/8$, hence $d(X_0 , A \otimes \mathfrak{s}_0) = d(X_0 , \mathfrak{s}_0) + n/4 - \varepsilon/2$.
\end{proof}

\begin{lemma}\label{lem:swrp}
Let $n,\varepsilon$ be as in Lemma \ref{lem:rpcov}. Then:
\begin{itemize}
\item[(1)]{If $n+2r = 0 \; ({\rm mod} \; 8)$ and $d(\widetilde{X}_0 , \widetilde{\mathfrak{s}}_0) \ge 0$, then
\[
SW(\widetilde{X}_0 , \widetilde{\mathfrak{s}}_0) = \binom{m_1-d_1}{\delta_0-m_0}SW(X , \mathfrak{s}) + \binom{m_0-d_0}{\delta_1 - m_1}SW(X_0 , A \otimes \mathfrak{s}_0) \; ({\rm mod} \; 2),
\]
where 
\[
\delta_0 = d(X , \mathfrak{s})/2, \quad \delta_1 = \delta_0 + (n - 2\varepsilon)/8,
\]
$d_i = \delta_i - (b_+(X_0)+1)/2$ and $m_0,m_1 \ge 0$ are any non-negative integers such that 
\[
m_0 + m_1 = \frac{d(\widetilde{X}_0 , \widetilde{\mathfrak{s}}_0)}{2} = 2\delta_0 + \frac{(r+n-\varepsilon)}{4}.
\]
}
\item[(2)]{If $n+2r = 4 \; ({\rm mod} \; 8)$ and $d(\widetilde{X}_0 , \widetilde{\mathfrak{s}}_0) > 0$, then
\[
SW_{\mathbb{Z}_2 , \widetilde{X}_0 , \widetilde{\mathfrak{s}}_0}( x^{m_0+m_1} ) = \left( \binom{m_1-d_1}{\delta_0-m_0}SW(X , \mathfrak{s}) + \binom{m_0-d_0}{\delta_1 - m_1}SW(X_0 , A \otimes \mathfrak{s}_0) \right)u
\]
where $u$ is the generator of $H^1_{\mathbb{Z}_2}(pt) \cong \mathbb{Z}_2$ and $m_0,m_1 \ge 0$ are any non-negative integers such that
\[
m_0 + m_1 = \frac{d(\widetilde{X}_0 , \widetilde{\mathfrak{s}}_0)+1}{2} = 2\delta_0 + \frac{(r+n+2-\varepsilon)}{4}.
\]
}
\item[(3)]{
If $n_i \ge -2$ for all $i$, then
\[
SW(\widetilde{X}_0 , \widetilde{\mathfrak{s}}_0) = SW(\widehat{X} , \widehat{\mathfrak{s}})
\]
for any sharp extension $\widehat{\mathfrak{s}}$ of $\widetilde{\mathfrak{s}}_0$.
}
\item[(4)]{
If $n_i \ge 0$ for all $i$, then
\[
SW(X_0 , A \otimes \mathfrak{s}_0) = SW(X' , \mathfrak{s}')
\]
for any sharp extension $\mathfrak{s}'$ of $A \otimes \mathfrak{s}_0$.
}
\end{itemize}
\end{lemma}
\begin{proof}
There are two lifts $s_0,s_1$ of $G = \mathbb{Z}_2$ to $\widetilde{\mathfrak{s}}_0$ and these correspond to the spin$^c$-structures $\mathfrak{s}_0, \mathfrak{s}_1 = A \otimes \mathfrak{s}_0$ on $X_0$. Then Proposition \ref{prop:free} gives
\[
\overline{SW}_{\mathbb{Z}_2}(\widetilde{X}_0 , \widetilde{\mathfrak{s}}_0)^{s_0} = SW(X_0 , \mathfrak{s}_0), \quad \overline{SW}_{G}(\widetilde{X}_0 , \widetilde{\mathfrak{s}}_0)^{s_1} = SW(X_0 , A \otimes \mathfrak{s}_0).
\]
Therefore, Theorem \ref{thm:eqre} gives:
\begin{align*}
& SW_{\mathbb{Z}_2 , \widetilde{X}_0 , \widetilde{\mathfrak{s}}_0}( x^{m_0}(x+v)^{m_1} ) \\
& \quad \quad = \left( \binom{m_1-d_1}{\delta_0-m_0}SW(X , \mathfrak{s}) + \binom{m_0-d_0}{\delta_1 - m_1}SW(X_0 , A \otimes \mathfrak{s}_0) \right)u^{2m_0+2m_1 - d(\widetilde{X}_0 , \widetilde{\mathfrak{s}}_0)}
\end{align*}
where $\delta_i = d(X_0 , \mathfrak{s}_i)/2$, $d_i = ind_{APS}(X_0 , \mathfrak{s}_i)$. Using Lemma \ref{lem:rpcov}, we have:
\begin{align*}
b_+(\widetilde{X}_0) &= 2b_+(X_0) + 1 - \frac{(n+2r)}{4} \\
d(\widetilde{X}_0 , \widetilde{\mathfrak{s}}_0) &= 4 \delta_0 + \frac{(r+n-\varepsilon)}{2} \\
\delta_1 &= \delta_0 + \frac{(n-2\varepsilon)}{8}.
\end{align*}
Also $\delta_i = d(X_0 , \mathfrak{s}_i)/2 = d_i - (b_+(X_0)+1)/2$. Hence (1) and (2) follow easily. (3) and (4) are immediate from Theorem \ref{thm:pscglue}.
\end{proof}

\begin{theorem}\label{thm:rpadj}
Let $X$ be a compact, oriented, smooth $4$-manifold with $b_+(X) > 1$ and $H_1(X ; \mathbb{Z}_2) = 0$. Suppose $S_1, \dots , S_r \subset X$ are disjoint smoothly embedded projective planes with Euler numbers $n_i = e(S_i)$. Suppose that any $r-1$ of $[S_1] , \dots , [S_r]$ are linearly independent in $H_2(X ; \mathbb{Z}_2)$ and that $[S_1] + \cdots + [S_r] = 0 \in H_2(X ; \mathbb{Z}_2)$. Suppose that $X$ that $SW(X , \mathfrak{s}) \neq 0 \; ({\rm mod} \; 2)$ for some spin$^c$-structure with $d(X , \mathfrak{s}) = 0$. Then the following hold:
\begin{itemize}
\item[(1)]{We have
\[
e(S_1) + \cdots + e(S_r) \le \max\{ 0 , 8-2r \}.
\]
}
\item[(2)]{If $r \le 3$, $e(S_i) \ge -2$ for all $i$ and $e(S_1) + \cdots + e(S_r) = 8-2r$, then $\widehat{X}$ does not have simple type.}
\end{itemize}
\end{theorem}
\begin{proof}
First note that $n+2r$ is divisible by $4$, by Lemma \ref{lem:rpcov} and integrality of $b_+(\widetilde{X}_0)$. Suppose that $n > 0$, $n+2r \ge 16$ and $n+2r = 0 \; ({\rm mod} \; 8)$. Since $d(X , \mathfrak{s}) \ge 0$, Lemma \ref{lem:swrp} (1) gives
\[
SW(\widetilde{X}_0 , \widetilde{\mathfrak{s}}_0) = \binom{m_1-d_1}{-m_0} + \binom{m_0-d_0}{\delta_1 - m_1}SW(X_0 , A \otimes \mathfrak{s}_0) \; ({\rm mod} \; 2)
\]
where $m_0,m_1$ are non-negative and $m_0 + m_1 = m = (r+n-\varepsilon)/4 \ge n/4 > 0$. Suppose that $SW(X_0 , A \otimes \mathfrak{s}_0) = 0 \; ({\rm mod} \; 2)$. Then $SW(\widetilde{X}_0 , \widetilde{\mathfrak{s}}_0) = \binom{m_1-d_1}{-m_0} \; ({\rm mod} \; 2)$. However, setting $m_0 = 0$ and $m_0 = 1$ gives two different values for $SW(\widetilde{X}_0 , \widetilde{\mathfrak{s}}_0)$. This is impossible so $SW(X_0 , A \otimes \mathfrak{s}_0)  = 1 \; ({\rm mod} \; 2)$ and hence
\[
SW(\widetilde{X}_0 , \widetilde{\mathfrak{s}}_0) = \binom{m_1-d_1}{-m_0} + \binom{m_0-d_0}{\delta_1 - m_1} \; ({\rm mod} \; 2).
\]
Suppose $m \ge \delta_1 + 2$. Then setting $(m_0 , m_1) = (0, m)$ or $(1,m-1)$ gives two different values for $SW(\widetilde{X}_0 , \widetilde{\mathfrak{s}}_0)$. This is impossible, so $m \le \delta_1 + 1$. We have $\delta_1 = (n - 2\varepsilon)/8$ and $m =  (r+n-\varepsilon)/4$, hence $-1 \le \delta_1 - m = -(n+2r)/8 \le -2$, a contradiction. So we can not have $n>0$, $n+2r \ge 16$ and $n+2r = 0 \; ({\rm mod} \; 8)$.

Suppose now that $n > 0$, $n+2r \ge 12$ and $n+2r = 4 \; ({\rm mod} \; 8)$. Lemma \ref{lem:swrp} (2) gives
\[
SW_{\mathbb{Z}_2 , \widetilde{X}_0 , \widetilde{\mathfrak{s}}_0}( x^{m_0+m_1} ) = \left( \binom{m_1-d_1}{-m_0} + \binom{m_0-d_0}{\delta_1 - m_1}SW(X_0 , A \otimes \mathfrak{s}_0) \right)u.
\]
where $m_0,m_1$ are non-negative and $m_0 + m_1 = m = (r+n+2-\varepsilon)/4 \ge (n+2)/4 > 0$. The same argument as above shows that $SW(X_0 , A \otimes \mathfrak{s}_0) = 1 \; ({\rm mod} \; 2)$ and that $m \le \delta_1 + 1$. Hence $-1 \le -(n+2r+4)/8 \le -2$, a contradiction. Hence either $n \le 0$ or $n+2r \le 8$. Therefore $n \le \max\{ 0 , 8-2r\}$, which proves (1).

Now suppose that $r \le 3$, $e(S_i) \ge -2$ for all $i$ and $n = e(S_1) + \cdots + e(S_r) = 8-2r$. Then as shown above we have
\[
SW(\widetilde{X}_0 , \widetilde{\mathfrak{s}}_0) = \binom{m_1-d_1}{-m_0} + \binom{m_0-d_0}{\delta_1 - m_1} \; ({\rm mod} \; 2).
\]
Also $m = \delta_1 + 1$ and $\delta_1 \ge 0$. Choosing $m_0=0$, $m_1 = \delta_1+1$ gives $SW(\widetilde{X}_0 , \widetilde{\mathfrak{s}}_0) = 1 \; ({\rm mod} \; 2)$. Then from Lemma \ref{lem:swrp} (3), we have $SW(\widehat{X} , \widehat{\mathfrak{s}}) = SW(\widetilde{X}_0 , \widetilde{\mathfrak{s}}_0) = 1 \; ({\rm mod} \; 2)$. Also $d(\widehat{X} , \widehat{\mathfrak{s}}) = d(\widetilde{X}_0 , \widetilde{\mathfrak{s}}_0) = 2m = \delta_1 + 1 > 0$, so $\widehat{X}$ does not have simple type.
\end{proof}

\begin{theorem}
Let $X$ be a compact, oriented, smooth $4$-manifold with $b_+(X) > 1$ and $H_1(X ; \mathbb{Z}_2) = 0$. Suppose $S \subset X$ is a smoothly embedded projective plane such that $[S] = 0 \in H_2(X ; \mathbb{Z}_2)$. Suppose that $SW(X , \mathfrak{s}) \neq 0 \; ({\rm mod} \; 2)$ for some spin$^c$-structure with $d(X , \mathfrak{s}) = 0$. Then 
\begin{itemize}
\item[(1)]{We have $e(S) \le 6$.}
\item[(2)]{If $b_+(X) = 3 \; ({\rm mod} \; 4)$, then $e(S) \le 2$.}
\item[(3)]{If $e(S) = 6$, then both $\widehat{X}$ and $X'$ do not have simple type.}
\end{itemize}
\end{theorem}
\begin{proof}
Part (1) is immediate from Theorem \ref{thm:rpadj} (1). Suppose $b_+(X) = 3 \; ({\rm mod} \; 4)$ and $e(S) > 2$. Then $e(S) = 6$, $n+2r = 8$ and by the same argument used in the proof of Theorem \ref{thm:rpadj}, we have $SW(X_0 , A \otimes \mathfrak{s}_0) = 1 \; ({\rm mod} \; 2)$. However, we also have $d(X_0 , A \otimes \mathfrak{s}_0) = (3 - \varepsilon)/2 > 0$. But this is impossible since $b_+(X) = 3 \; ({\rm mod} \; 4)$ implies that if the Seiberg--Witten invariant of a spin$^c$-structure $\mathfrak{s}'$ on $X_0$ is odd, then $d(X_0 , \mathfrak{s}') = 0$ (see \cite[Theorem 3.7]{bf} for the case of a closed $4$-manifold. Alternatively see \cite[Theorem 1.8]{bar} where the result is given for abstract monopole maps). Hence $e(S) \le 2$. If $e(S) = 6$ and $b_+(X) = 1 \; ({\rm mod} \; 4)$ then the same calculation shows that $\widehat{X}$ and $X'$ do not have simple type.
\end{proof}

\section{Embedded spheres with cusps}\label{sec:cusp}

Let $X$ be a smooth $4$-manifold. By an embedded sphere in $X$ with a cusp, we mean a $2$-sphere $C$ and a smooth injective map $i : C \to X$ such that $i$ is an embedding except at a single point $c \in C$, the cusp, where locally $C$ is given by the cone over a $(2,3)$-torus knot $T \subset S^3$. We will say $C$ is a {\em left-handed cusp} or {\em right-handed cusp} according to whether $T$ is a left or right-handed torus knot. Our notation is that $T = T_{2,3}$ denotes the right-handed $(2,3)$-torus knot and $-T_{2,3}$ the left-hand $(2,3)$-torus knot.

An embedded cusp has a fundamental class $[C] \in H_2( X ; \mathbb{Z})$ which is the image under $i$ of the fundamental class of $C$. If $X$ is compact and oriented we will identify $[C]$ with its Poincar\'e dual class $[C] \in H^2(X ; \mathbb{Z})$. Let $n = [C]^2$ be the self-intersection number. $C$ can be constructed as the union of the cone over $T$ and a $2$-cell which has self-intersection $n$ relative the Seifert framing on $T$. It follows that a closed neighbourhood of $C$ in $X$ can be identified with the $4$-manifold $X_{n}(T)$, the trace of the $n$-surgery on $T$. The boundary of $X_n(T)$ is $S_n(T)$, the $n$-surgery on $T$.

\begin{proposition}\label{prop:E8-p}
Let $1 \le p \le 7$. 
\begin{itemize}
\item[(1)]{$S_{-p}(-T_{2,3})$ has positive scalar curvature and bounds a negative almost rational plumbing (or a boundary connected sum of two negative definite plumbings in the case $p=6$).}
\item[(2)]{Suppose $C \subset X$ is a left-handed cusp with self-intersection $-p$. Let $\mathfrak{s}$ be a spin$^c$-structure on $X$ for which $| \langle c(\mathfrak{s}) , [C] \rangle | + [C]^2 = 0$. Then $\mathfrak{s}|_{X_{-p}}(-T_{2,3})$ is not sharp in the sense of Definition \ref{def:sharp}.}
\end{itemize}
\end{proposition}
\begin{proof}
The result of surgery on a torus knot was calculated in \cite{mo}, see also \cite{kash} for another proof in which the orientation of the resulting manifold is more apparent. For $1 \le p \le 5$ we get that $S_{-p}(-T_{2,3}) = M( 1 \, ; (2,1) , (3,1) , (6-p , 1))$ (where our notation for Seifert manifolds is taken from \cite[\textsection 1.1.4]{sav}), which is the boundary of the negative definite $E_{(9-p)}$-plumbing (where $E_5 = D_5$, $E_4 = A_4$). When $p=6$, we instead get $S_{-6}(-T_{2,3}) = L(2,1) \# L(3,2)$, which is the boundary of the boundary connected sum of two negative definite plumbings (of types $A_1$ and $A_2$). For $p=7$, we get $S_{-7}(-T_{2,3}) = L(7,2)$, which is a lens space hence bounds a negative definite almost rational plumbing.

Now suppose that $C \subset X$ is a left-handed cusp with self-intersection $-p$. Let $\mathfrak{s}$ be a spin$^c$-structure on $X$ for which $| \langle c(\mathfrak{s}) , [C] \rangle | + [C]^2 = 0$. Let $\mathfrak{s}_Y = \mathfrak{s}|_Y$ be the restriction of $\mathfrak{s}$ to $Y = \partial X_{-p}(-T_{2,3}) = S_{-p}(-T_{2,3})$. We need to show that $\delta( X_{-p}(-T_{2,3}) , \mathfrak{s}|_{X_{-p}(-T_{2,3})} ) < \delta( Y , \mathfrak{s}_Y)$. Since $| \langle c(\mathfrak{s}) , [C] \rangle | = -[C]^2 = p$, it follows that 
\[
\delta( X_{-p}(-T_{2,3}) , \mathfrak{s}|_{X_{-p}(-T_{2,3})} ) = -\frac{p^2}{8p} + \frac{1}{8} = -\frac{(p-1)}{8}.
\]
Let $\mathfrak{s}'_Y$ be the spin-structure on $Y$ defined as follows. Let $W_p$ denote the negative definite $E_{(9-p)}$-plumbing if $p \neq 6,7$. Define $W_6$ to be the boundary connected sum of the $A_1$ and $A_2$ plumbings and define $W_7$ to be the plumbing on the graph with two vertices of degrees $-8,-2$ and an edge joining them. Then $Y = \partial W_p$. Since all vertices in the plumbing has even degree, there is a (unique) spin$^c$-structure $\mathfrak{s}_W$ on $W$ with $c( \mathfrak{s}_W) = 0$. Let $\mathfrak{s}'_Y = \mathfrak{s}_W |_Y$. Since $b_2(W) = 9-p$, we find that
\[
\delta( Y , \mathfrak{s}'_Y ) = \delta(W , \mathfrak{s}_W) = \frac{(9-p)}{8} = 1 - \frac{(p-1)}{8}.
\]
Then (2) will follow if we can show that $\mathfrak{s}'_Y = \mathfrak{s}_Y$. To see this we first note that every spin $^c$-structure $\mathfrak{s}''_Y$ on $Y$ is the restriction of a spin$^c$-structure $\mathfrak{s}''$ on $X_{-p}$ such that $|c''| \le p$, where $c'' = \langle c(\mathfrak{s}'') , [C] \rangle$. Since $c(\mathfrak{s}'')$ is characteristic it follows that $c'' = p \; ({\rm mod} \; 2)$. Furthermore, since $ind_{APS}( X_{-p}(-T_{2,3}) , \mathfrak{s}'') = \delta( X_{-p})(-T_{2,3}) , \mathfrak{s}'') - \delta(Y , \mathfrak{s}''_Y) \in \mathbb{Z}$, we have
\[
\delta(Y , \mathfrak{s}''_Y) = \delta(X_{-p}(-T_{2,3}) , \mathfrak{s}'') = -\frac{(c'')^2}{8p} + \frac{1}{8} \; ({\rm mod} \; \mathbb{Z}).
\]
One finds by direct calculation that if $-(c'')^2/8p + 1/8 = \delta(Y , \mathfrak{s}'_Y) \; ({\rm mod} \; \mathbb{Z})$, where $|c''| \le p$, $c'' = p \; ({\rm mod} \; 2)$, $1 \le p \le 7$, then $|c''| = p$. But $|c''| = p$ is precisely the case where $\mathfrak{s}''_Y = \mathfrak{s}_Y$. Hence $\mathfrak{s}_Y = \mathfrak{s}'_Y$ and the result is proven.
\end{proof}

\begin{proposition}\label{prop:lcusp}
Let $X$ be a compact, oriented, smooth $4$-manifold with $b_1(X) = 0$, $b_+(X) > 1$ and suppose $X$ has simple type. Let $C \subseteq X$ be an embedded sphere with a cusp. Suppose that $SW(X , \mathfrak{s}) \neq 0$ for some spin$^c$-structure $\mathfrak{s}$.
\begin{itemize}
\item[(1)]{We have $| \langle c(\mathfrak{s}) , [C] \rangle | + [C]^2 \le 0$.}
\item[(2)]{If $C$ has a left-handed cusp, $-7 \le [C]^2 \le -1$ and $| \langle c(\mathfrak{s}) , [C] \rangle | + [C]^2 = 0$, then the $4$-manifold $X'$ obtained by removing a neighbourhood of $C$ and attaching the plumbing $W_{-[C]^2}$ as in Proposition \ref{prop:E8-p} does not have simple type.}
\end{itemize}

\end{proposition}
\begin{proof}
By removing a neighbourhood of the cusp of $C$ and replacing it by a genus $1$ surface with one boundary component, we obtain a smoothly embedded torus $C' \subset X$ with $[C'] = [C]$. Since $X$ has simple type the adjunction inequality says that $| \langle c(\mathfrak{s}) , [C] \rangle | + [C]^2 \le 0$.

Now suppose $C$ has a left-handed cusp and $1 \le p \le 7$, where $p = -[C]^2$. Since $| \langle c(\mathfrak{s}) , [C] \rangle | = p$, we have that $\mathfrak{s}_{X_{-p}(-T_{2,3})}$ is not sharp by, Proposition \ref{prop:E8-p} (2). Now let $X_0$ be $X$ with a neighbourhood of $C$ removed and $X' = W_p \cup_Y X_0$, $Y = S_{-p}(-T_{2,3})$. Let $\mathfrak{s}_0 = \mathfrak{s}|_{X_0}$ and let $\mathfrak{s}'$ be a sharp extension of $\mathfrak{s}_0$. Then $SW(X' , \mathfrak{s}') = SW(X , \mathfrak{s}) \neq 0$ by Theorem \ref{thm:pscglue} and $d(X' , \mathfrak{s}') > d(X , \mathfrak{s}) = 0$ since $\mathfrak{s}|_{X_{-p}(-T_{2,3})}$ is not sharp. Hence $X'$ does not have simple type.
\end{proof}

\begin{corollary}\label{cor:lcusp0}
Let $X$ be a compact, oriented, smooth $4$-manifold with $b_1(X) = 0$, $b_+(X) > 1$ and suppose $X$ has simple type. Let $C \subseteq X$ be an embedded sphere with a left-handed cusp and $[C]^2 = 0$. Suppose that $SW(X , \mathfrak{s}) \neq 0$ for some spin$^c$-structure $\mathfrak{s}$. Let $X'$ be the $4$-manifold obtained from $X \# \overline{\mathbb{CP}^2}$ by removing a neighbourhood of $C \# E$ (where $E \subset \overline{\mathbb{CP}^2}$ is a sphere of self-intersection $-1$) and attaching the $E_8$ plumbing. Then $X'$ does not have simple type. If $SW(X , \mathfrak{s}) \neq 0 \; ({\rm mod} \; 2)$ and $b_+(X) = 3 \; ({\rm mod} \; 4)$ then no such $C$ exists.
\end{corollary}
\begin{proof}
Since $C \# E$ is a left-handed cusp with self-intersection $-1$, it follows form Proposition \ref{prop:lcusp} (2) that $X'$ does not have simple type. Suppose that $SW(X,\mathfrak{s}) \neq 0 \; ({\rm mod} \; 2)$. Then $X'$ has a spin$^c$-structure $\mathfrak{s}'$ with $SW(X , \mathfrak{s}') = SW(X , \mathfrak{s}) \neq 0 \; ({\rm mod} \; 2)$ and $d(X , \mathfrak{s}') > 0$. Then we must have  $b_+(X') = 1 \; ({\rm mod} \; 4)$ (using \cite[Theorem 1.8]{bar}). But $b_+(X') = b_+(X)$, so if $b_+(X) = 3 \; ({\rm mod} \; 4)$ then no such $C$ exists.
\end{proof}

\begin{corollary}
Let $X$ be a compact, oriented, smooth $4$-manifold with $b_1(X) = 0$, $b_+(X) > 1$ and suppose $X$ has simple type. Let $C \subseteq X$ be an embedded sphere with a cusp. Suppose that $SW(X , \mathfrak{s}) \neq 0 \; ({\rm mod} \; 2)$ for some spin$^c$-structure $\mathfrak{s}$ and that $b_+(X) = 3 \; ({\rm mod} \; 4)$. If $[C]^2 \ge -7$, then $| \langle c(\mathfrak{s}) , [C] \rangle | + [C]^2 < 0$.
\end{corollary}
\begin{proof}
The argument is similar to the proof of Corollary \ref{cor:lcusp0}. 
\end{proof}

We finish this section by noting that there is a close relationship between embedded spheres with cusps and embedded projective planes. Let $C$ be an embedded sphere with a cusp. The link of the singularity is the $(2,3)$-torus knot, which bounds a M\"obius band of self-intersection $-6$ (for right-handed torus knots) or $6$ (for left-handed torus knots). Hence by removing a neighbourhood of the cusp and attaching such a M\"obius band, we obtain an embedded projective plane $S \subset X$ such that $[S] = [C] \; ({\rm mod} \; 2)$ and $e(S) = -6 + [C]^2$ for a right-handed cusp, or $e(S) = 6 + [C]^2$ for a left-handed cusp. Using this, the results of Section \ref{sec:rp2} imply corresponding results for embedded spheres with cusps. For instance, Theorem \ref{thm:rpadj} implies the following:

\begin{theorem}
Let $X$ be a compact, oriented, smooth $4$-manifold with $b_+(X) > 1$ and $H_1(X ; \mathbb{Z}_2) = 0$. Suppose $C_1, \dots , C_r \subset X$ are disjoint smoothly embedded spheres with left-handed cusps. Suppose that any $r-1$ of $[C_1] , \dots , [C_r]$ are linearly independent in $H_2(X ; \mathbb{Z}_2)$ and that $[C_1] + \cdots + [C_r] = 0 \in H_2(X ; \mathbb{Z}_2)$. Suppose that $X$ that $SW(X , \mathfrak{s}) \neq 0 \; ({\rm mod} \; 2)$ for some spin$^c$-structure with $d(X , \mathfrak{s}) = 0$. Then
\[
[C_1]^2 + \cdots + [C_r]^2 \le \max\{ -6r , 8-8r \}.
\]
\end{theorem}


\bibliographystyle{amsplain}

\begin{thebibliography}{99}
\bibitem{bakl}R. H. Bamler, B. Kleiner, Ricci flow and contractibility of spaces of metrics, arXiv:1909.08710 (2019).
\bibitem{bar}D. Baraglia, Equivariant Seiberg--Witten theory, arXiv:2406.00642 (2024).
\bibitem{bako}D. Baraglia, H. Konno, On the Bauer-Furuta and Seiberg-Witten invariants of families of $4$-manifolds. {\em J. Topol.} {\bf 15} (2022), no. 2, 505-586. 
\bibitem{bh}D. Baraglia, P. Hekmati, Equivariant Seiberg--Witten--Floer cohomology. {\em Algebr. Geom. Topol.} {\bf 24} (2024), no. 1, 493-554.
\bibitem{bf}S. Bauer, M. Furuta, A stable cohomotopy refinement of Seiberg--Witten invariants. I. {\em Invent. Math.} {\bf 155} (2004), no. 1, 1-19.
\bibitem{don}H. Donnelly, Eta invariants for $G$-spaces. {\em Indiana Univ. Math. J.} {\bf 27} (1978), no. 6, 889-918. 
\bibitem{fs}R. Fintushel, R. J. Stern, Immersed spheres in $4$-manifolds and the immersed Thom conjecture. {\em Turkish J. Math.} {\bf 19} (1995), no. 2, 145-157.
\bibitem{gm}L. Guillou, A. Marin, {\em Une extension d'un th\'eor\`eme de Rohlin sur la signature}. \`A la recherche de la topologie perdue, 97-118, Progr. Math., {\bf 62}, Birkh\"auser Boston, Boston, MA, (1986). 
\bibitem{kash}T. Kadokami, M. Shimozawa, Dehn surgery along torus links. {\em J. Knot Theory Ramifications} {\bf 19} (2010), no. 4, 489-502. 
\bibitem{law}T. Lawson, Normal bundles for an embedded $\mathbb{RP}^2$ in a positive definite $4$-manifold. {\em J. Differential Geom.} {\bf 22} (1985), no. 2, 215-231. 
\bibitem{lrs}A. S. Levine, D. Ruberman, S. Strle, Nonorientable surfaces in homology cobordisms. With an appendix by I. M. Gessel. {\em Geom. Topol.} {\bf 19} (2015), no. 1, 439-494. 
\bibitem{man1}C. Manolescu, Seiberg-Witten-Floer stable homotopy type of three-manifolds with $b_1=0$. {\em Geom. Topol.} {\bf 7} (2003), 889-932. 
\bibitem{man2}C. Manolescu, A gluing theorem for the relative Bauer-Furuta invariants. {\em J. Differential Geom.} {\bf 76} (2007), no. 1, 117-153. 
\bibitem{mo}L. Moser, Elementary surgery along a torus knot. {\em Pacific J. Math.} {\bf 38} (1971), 737-745. 
\bibitem{nem}A. N\'emethi, On the Ozsv\'ath--Szab\'o invariant of negative definite plumbed 3-manifolds. {\em Geom. Topol.} {\bf 9} (2005), 991-1042.
\bibitem{or}P. Orlik, {\em Seifert manifolds}. Lecture Notes in Mathematics, Vol. {\bf 291}. Springer-Verlag, Berlin-New York, (1972). viii+155 pp. 
\bibitem{os2}P. Ozsv\'ath, Z. Szab\'o, The symplectic Thom conjecture. {\em Ann. of Math.} (2) {\bf 151} (2000), no. 1, 93-124.
\bibitem{park}B. D. Park, Seiberg-Witten invariants and branched covers along tori. {\em Proc. Amer. Math. Soc.} {\bf 133} (2005), no. 9, 2795-2803. 
\bibitem{rok}V. A. Rokhlin, Two-dimensional submanifolds of four-dimensional manifolds. {\em Funkcional. Anal. i Prilozen.} {\bf 5} (1971), no. 1, 48-60. 
\bibitem{ruwa}Y. Ruan, S. Wang, Seiberg-Witten invariants and double covers of $4$-manifolds. {\em Comm. Anal. Geom.} {\bf 8} (2000), no. 3, 477-515. 
\bibitem{sav}N. Saveliev, {\em Invariants for homology 3-spheres}. Encyclopaedia of Mathematical Sciences, {\bf 140}. Low-Dimensional Topology, I. Springer-Verlag, Berlin, (2002). xii+223.
\end{thebibliography}

\end{document}